\documentclass[11pt,letterpaper]{amsart}
\usepackage[utf8]{inputenc}

\usepackage{amssymb}
\usepackage[english]{babel}
\usepackage[mathscr]{eucal}
\usepackage[all,cmtip]{xy}
\usepackage{appendix}

\usepackage{hyperref}
\usepackage{enumitem}

\usepackage{color}

\newcommand{\ZZ}{{\mathbb Z}}

\newcommand{\RR}{{\mathbb R}}

\newcommand{\QQ}{{\mathbb Q}}
\newcommand{\NN}{{\mathbb N}}

\newcommand{\stab}{\mathrm{Stab}}

\newcommand{\Ga}{{\mathbb{G}_a}}

\newcommand{\R}{{\mathbb R}}

\newcommand{\A}{\mathbb{A}}
\newcommand{\rk}{\mathrm{rk}}

\newcommand{\cO}{\mathscr{O}}

\renewcommand{\u}[1]{\boldsymbol{#1}}
\usepackage{mathtools}
\usepackage{scalerel}
\DeclarePairedDelimiter{\norm}{\lVert}{\rVert}

\newtheorem{theo}{Theorem}[section]

\newtheorem{lemma}[theo]{Lemma}
\newtheorem{prop}[theo]{Proposition}
\newtheorem{coro}[theo]{Corollary}
\newtheorem{rema}[theo]{Remark}
\newtheorem{exam}[theo]{Example}

\newtheorem{defi}[theo]{Definition}

\newcommand{\GL}{\mathrm{GL}}
\newcommand{\PEP}{\mathrm{PEP}}
\newcommand{\SL}{\mathrm{SL}}

.240pk scaled 1200 .240pk
\begin{document}
	\title[PEP and applications]{Purely exponential parametrizations and their group-theoretic applications}
	
	\author[P.~Corvaja]{Pietro Corvaja}
	\author[J.~Demeio]{Julian L. Demeio}
	\author[A.~Rapinchuk]{Andrei S. Rapinchuk}
	\author[J.~Ren]{Jinbo Ren}
	\author[U.~Zannier]{Umberto M. Zannier}

	\address{Dipartimento di Scienze Matematiche, Informatiche e Fisiche, via delle Scienze,
		206, 33100 Udine, Italy}
	
	\email{pietro.corvaja@uniud.it}
	
	\address{Department of Mathematical Sciences, 6 West, University of Bath, Claverton Down, Bath, BG2 7AY, United Kingdom}
	
	\email{demeiojulian@yahoo.it}
	
	\address{Department of Mathematics, University of Virginia,
		Charlottesville, VA 22904-4137, USA}
	
	\email{asr3x@virginia.edu}

	\address{School of Mathematical Sciences, Xiamen University, Xiamen 361005, China}
	
	\email{jren@xmu.edu.cn}

	\address{Scuola Normale Superiore, Piazza dei Cavalieri, 7, 56126 Pisa, Italy}
	
	\email{umberto.zannier@sns.it}

	\begin{abstract}

This paper is mainly motivated by the analysis of the so-called {\em Bounded Generation property} (BG) of linear groups (in characteristic $0$), which is known to admit far-reaching group-theoretic implications.

We achieve complete answers to certain longstanding open questions about Bounded Generation (sharpening considerably some results of \cite{CRRZ}). For instance, we prove that {\em linear groups boundedly generated by semi-simple elements are necessarily virtually abelian}. This is obtained as a corollary of sparseness of subsets which are likewise generated.

In the paper in fact we go further, framing (BG) in the more general context of {\em (Purely) Exponential Parametrizations} (PEP) for subsets of affine spaces, a concept which unifies different issues. Using deep tools from Diophantine Geometry (including the Subspace Theorem), we systematically develop a theory showing in particular that for a (PEP) set over a number field, the asymptotic distribution of its points of Height at most $T$ is always $\sim c(\log T)^r$, with certain constants ${c}>0$ and ${r}\in \ZZ_{\geq 0}$. (This shape fits with a well-known viewpoint first put forward by Manin.)
  	
	\end{abstract}
	
	\maketitle
	
	\tableofcontents
	
	\section{Introduction}

The so-called {\bf Bounded Generation property} (BG) for linear groups, i.e.\ expressibility of the group as a product of cyclic subgroups, is well-known to have far-reaching consequences in several directions, sometimes toward fundamental problems; hence it is highly needed to have general criteria for this property to hold or not. In the recent paper \cite{CRRZ} a longstanding question was answered, by showing in particular that (in characteristic $0$) a boundedly generated {\em anisotropic} linear group is necessarily virtually abelian. \footnote{See below for reminders on some group-theoretic definitions.}

However a number of other natural questions were left open in \cite{CRRZ}, and were seemingly lying outside the range of the method used therein. For instance, ``Question B" in \cite{CRRZ}\footnote{In fact, this is the second ``Question A'' on page 3 of the published version of \cite{CRRZ}.}, asked whether the ``{\em anisotropic}'' condition could be removed, weakening the requirement by assuming instead that (BG) {\em holds with semi-simple elements}. This ``Question B'' will be answered in full in this paper.

Also, no information was provided as to the distribution of Boundedly Generated subsets of a non-virtually abelian linear group, namely toward the natural question of {\em measuring the distance from the validity of property (BG) in the cases when it cannot be fully achieved}. In this paper we give rather complete answers to these issues (some of which were already announced in \cite{CDRRZcr}).

Moreover we frame Bounded Generation in the more general context of {\bf Purely Exponential Parametrizations} (PEP) for sets in affine spaces over fields of characteristic zero, which unifies different needs and issues.

Using deep tools from Diophantine Geometry, i.e.\ Schlickewei-Schmidt's subspace theorem, we systematically develop a theory showing in particular that for a (PEP) set over a number field, the asymptotic distribution of its points in terms of the Height $T$ is always exactly of the form $\sim c(\log T)^r$, with certain constants ${c}$ and ${r}$.

As a consequence for linear groups over number fields, we prove {\em sparseness of subsets boundedly generated by semi-simple elements}. In turn, this yields a complete characterisation of likewise generated linear groups, proving that they are necessarily virtually abelian.

    Before giving the precise statements of the main results, we start with the formal definitions. 
	\vskip2mm
	
	\noindent {\bf Definition.} {\rm Let $K$ be a field. A purely exponential polynomial is a function $f:\ZZ^r \to \overline{K}$ expressible as a sum:
		\begin{equation}\label{Eq:expressionPEP}
		f(\mathbf{x})=\sum_{j=1}^e a_j \lambda_1^{l_{1,j}(\mathbf{x})}\cdots \lambda_k^{l_{k,j}(\mathbf{x})},
		\end{equation}
		for certain constants $a_1,\dots,a_e,\lambda_1,\dots, \lambda_k\in \overline{K}^{*}$ and linear forms $l_{i,j}(\mathbf{x})$ in $r$ variables whose coefficients are {\bf rational integers}. We refer to the elements $\lambda_1,\dots,\lambda_k$ as the {\bf bases} of $f$, to the linear forms $l_{i,j}$ as the {\bf exponents} of ${f}$, and to the constants $a_j$ as the {\bf coefficients} of ${f}$. Every $f$ as above with $e=1$ is called a {\bf (purely exponential) monomial}. A vector-valued function $\mathbf{f}=(f_1,\dots,f_s)\colon \ZZ^r \to \overline{K}^s$ such that each component $f_i$ is a purely exponential polynomial is called a {\bf (vector of) purely exponential polynomial(s)} (PEP). The {\bf bases/exponents/coefficients} of such $\mathbf{f}$ are defined to be the union of those of the $f_i$'s.}
	
	\vskip2mm
	
	\noindent {\bf Definition.} {\rm Let $\Sigma$ be a subset of $\mathbb{A}_K^s$ ($K$ is a field). Then $\Sigma$ is said to have {\bf Purely Exponential Parametrization} (PEP) in $r$ variables if $\Sigma$ has the shape
		$$\Sigma=\Big\{ (f_1(\mathbf{n}),\dots,f_s(\mathbf{n}))\colon \mathbf{n}\in \mathbb{Z}^r \Big\}= \mathbf{f}(\mathbb{Z}^r),$$
		where each $f_i(\mathbf{x})=f(x_1,\dots, x_r)$ is a {\bf Purely Exponential Polynomial}. We also say that $\Sigma$ is a (PEP) set.}
	\vskip2mm
	
	In the cases we investigate, a $\Sigma$ as above will be often contained in an algebraic subvariety $V\subset \mathbb{A}_K^s$.

    \begin{rema}{\rm 
		It is easily seen (cf. Proposition \ref{Prop:Union}) that any finite union of (PEP) sets is still a (PEP) set. Moreover, the image of a (PEP) set under a polynomial map is also (PEP).}
	\end{rema}

 \begin{rema}\label{pepchar} {\rm  An equivalent way to define a purely exponential polynomial is to say that it is a \underline{linear combination of characters of $\ZZ^r$}, see for example \cite{SchlickeweiSchmidt93,SchlickeweiSchmidt00}. It follows from  Artin's linear independence of characters that this linear combination is uniquely determined by the polynomial.}
	\end{rema}
	\vskip2mm
	\noindent
	Note that, despite the uniqueness described in Remark \ref{pepchar}, the parameters $a_j$'s, $\lambda_i$'s and forms $l_{i,j}$'s in \eqref{Eq:expressionPEP} are instead not uniquely determined by $f$ (for example $2^n=(\sqrt{2})^{2n}$, or $2^n3^m=2^{n-m}6^m$). For convenience, in the rest of the paper, whenever we talk about a vector of purely exponential polynomials $\mathbf{f}=(f_1,\dots,f_s)$, we always fix a choice of these parameters and forms. This choice is often implicit and clear from context. See also the ``Notations and Conventions'' at the end of this section.

	\begin{rema}{\rm The more classical question of finding suitable {\em algebraic} parametrizations, i.e.\  what one obtains by
		replacing the
		$f_i(\mathbf{x})$'s above by usual multi-variable polynomials,  of a specific subset of an algebraic variety (e.g.\ \ the set of integral or rational points) has a long history. For example, one has the well-known parametrization of the integer solutions of the classical Pythagorean equation as $$(2tab, t(a^2-b^2),t(a^2+b^2))\text{ or }(t(a^2-b^2),2tab, t(a^2+b^2)),a,b,t\in \mathbb{Z}.$$ An extensive number of cases of the classical algebraic parametrization have been studied, see for example \cite{Zannier96},  \cite{Zannier03},\cite{Vasser} and \cite{LarsenNguyen} for some recent discoveries concerning algebraic parametrization of some arithmetic groups.}
	\end{rema}
	
	Examples of (PEP) sets of particular interest arise in the analysis of bounded generation (BG) subsets  of a linear group. Recall that an abstract group $\Gamma$ is said to have the  {\bf bounded generation} property (BG) if it can be written in the form
	\begin{equation} \label{bgequation}
	    \Gamma=\langle \gamma_1 \rangle \cdots \langle \gamma_r \rangle
	\end{equation}
	
	for certain fixed $\gamma_1,\dots, \gamma_r\in \Gamma$, where $\langle \gamma_i \rangle$ is the cyclic subgroup generated by $\gamma_i$.
	\begin{exam}{\rm
		Let $\Gamma \subset \mathrm{GL}_n (K)$ be a linear group which has (BG) by semi-simple elements, i.e.\ admits a presentation (\ref{bgequation}) with all the $\gamma_i$'s semi-simple. Writing
		$$g_i^{-1}\gamma_i g_i=\mathrm{diag}(\lambda_{i,1},\dots, \lambda_{i,n}),\text{ for all }i=1,\dots,r$$
		for certain $g_i\in \mathrm{GL}_n(\overline{K})$ and $\lambda_{i,j}\in \overline{K}^*$ for $i=1,\dots, r,j=1,\dots,n$, we see that every $\gamma \in \Gamma$ can be written in the form
		$$\gamma=\prod_{i=1}^r g_i\left[\mathrm{diag}(\lambda_{i,1}^{a_i},\dots, \lambda_{i,n}^{a_i})\right]g_i^{-1}\text{ for some }a_1,\dots,a_r\in \mathbb{Z}.$$
		The last equation gives rise to a (PEP) parametrization of $\Sigma$ as a subset of $M_n(K)=\mathbb{A}_K^{n^2}$ in $r$ variables with bases equal to the eigenvalues $\lambda_{i,j}$'s. }
	\end{exam}

    More examples of (PEP) can be found in \S \ref{generalitypep}.
 
	Groups with bounded generation were the starting point of this project. We refer the readers to \cite[\S 1]{CRRZ} for a discussion of a large number of applications of bounded generation. In particular: to rigidity, the congruence subgroup problem, the Margulis-Zimmer conjecture, etc. At the same time, since the main theorem of \cite{CRRZ} implies that infinite $S$-arithmetic subgroups of simple anisotropic algebraic groups are {\bf not} boundedly generated, one is interested in finding other abstract conditions with the same consequences of bounded generation that can potentially hold in the anisotropic situation. See \cite[\S 6]{CRRZ} for the discussion of the ``polynomial index growth'' condition. These issues, however, will not be discussed further in the current paper.

	In the current article, we will study the distribution of points of (PEP) subsets of $\mathbb{A}_K^n$ over a \underline{number field} $K$, with respect to the affine height $H_{\text{aff}}$. We will recall some facts about heights in \S \ref{heightdefprop}, here we only mention that for $\xi=p/q\in \mathbb{Q}$ with $p,q\in \ZZ$ coprime, we have $H_{\mathrm{aff}}(\xi)=\max \{|p|,|q|\}$. The main quantitative result of this paper reads
	\begin{theo}
	\label{firstmainthm}
	Let $K$ be a number field, $\Sigma \subset \mathbb{A}_K^n$ be a subset that has a (PEP) parametrization in $r$ variables. 
	Then there exists an integer $0 \leq r' \leq r$ and a(n explicitly computable) constant $c>0$ such that
		$$\#\{P\in \Sigma\colon H_{\text{aff}}(P)\leq H\}\sim c(\log H)^{r'}\text{ when }H\to \infty.$$ 
	\end{theo}
	
	In other words, any (PEP) set has \underline{exactly} {\bf logarithmic-to-the-}$r'$ distribution in terms of the height.

    The constant $c$ in Theorem \ref{firstmainthm} is expressible in terms of the volume of a certain norm $1$ ball, see Remark \ref{Rmk:meaningc} (one can compare this with the interpretation of the $C$ in \cite[Corollary 1.1]{Maucourant07}).

    The proof of Theorem \ref{firstmainthm} makes use of a height inequality of Evertse \cite{evertse84}, whose proof in turn relies on the version of the subspace theorem due to Schlickewei-Schmidt cf. \cite[Theorem 2.3]{CorvajaZannier}.

    In the literature, there are several asymptotics of the shape $c H^{\alpha}(\log H)^{\beta}$ predicted by various Manin-type conjectures for the counting of rational and integral points on certain varieties (see e.g.\ \ \cite{GMO, STT07}). Taking $\alpha = 0$, we may view Theorem  \ref{firstmainthm} as a Manin's conjecture type counting result.

    \begin{defi} {\rm We call the $r'$ appearing in Theorem \ref{firstmainthm} the {\bf rank} of $\Sigma$, denoted by $\rk_{\mathrm{PEP}}\Sigma$. }
    \end{defi}
    When there is no risk of confusion with the rank of an abelian group, the rank of a (PEP) set $\Sigma$ will be denoted by $\rk \Sigma$. The name ``rank'' here will be justified in Lemma \ref{Lemrankequalrank}.
 
	Theorem \ref{firstmainthm} strengthens the main quantitative result of \cite{CDRRZcr}, where it was only verified that the distribution of points of a (PEP) set is \underline{at most} logarithmic-to-the-$r$. Going back to the study of bounded generation, the initial goal of our project, letting $K$ be a number field, the above theorem tells us that for any semi-simple elements $\gamma_1,\dots, \gamma_r\in \mathrm{GL}_n(K)$, there is an integer $0\leq r'\leq r$ and a constant $c>0$ with
	\begin{equation}\label{BG}
	\# \{P\in \langle \gamma_1\rangle\cdots \langle \gamma_r\rangle \colon H_{\text{aff}}(P)\leq H\}\sim c(\log H)^{r'} \text{ when }H\to \infty.
	\end{equation}
	 When the linear group itself has (BG) by semi-simple elements, the value of $r'$ in (\ref{BG}) will be specified in the appendix, see Remark \ref{valuerprimebg}. Moreover, we mention here that similar counting problems have been considered by Everest-Shparlinski \cite{ES99}, in which they obtain some asymptotic results in terms of norms instead of heights. See also a more recent relevant preprint by Ostafe-Shparlinski \cite{OS22}.

	Recall that, see e.g.\  \cite[Example 1.6]{DukeRudnickSarnak}, we have $\#\left\{P\in \mathrm{SL}_n(\mathbb{Z})\colon H_{\mathrm{aff}}(P)\leq H\right\}\sim c'H^{n^2-n}$ for some $c'>0$. Combined with Theorem \ref{firstmainthm}, we obtain a quantitative proof that $\mathrm{SL}_n(\mathbb{Z})$ is not boundedly generated by semi-simple matrices (while it does have bounded generation by elementary matrices when $g\geq 3$, cf. \cite{CK83}). Originally this result was established in \cite[Theorem 1.1]{CRRZ}.

	Let $K$ be a number field, and $S$ be a finite set of places containing all infinite ones. We let $\cO_S$ denote the ring of $S$-integers. Let $G\subset \mathrm{GL}_{n,K}$ be a linear algebraic group over $K$, and let $G^{\mathrm{der}}$ be the derived group of $G$. 
    
    For a linear algebraic $K$-group $G$, we call a subgroup $\Gamma < G(K)$ {\bf $S$-arithmetic} if $\Gamma$ is commensurable with $G(\cO_S)\coloneqq G(K) \cap \GL_n(\cO_S)$. This definition is actually independent of the chosen embedding $G\hookrightarrow \mathrm{GL}_{n,K}$ (see \cite[\S 5.4]{PR}). We signal to the interested reader that a more general definition has been given in \cite[\S 1]{PrRIHES}, but adopting it would go beyond the scope of our paper.
     
     The following counting property of $S$-arithmetic groups is a consequence of numerous counting results on lattice points in algebraic groups, combined with some straightforward volume estimates.
 
	\begin{theo}\label{volestimate} Let $G$ be a connected linear algebraic group over a number field $K$, and $\Gamma$ be an $S$-arithmetic subgroup of $G$. Let $\rho \colon G\hookrightarrow \mathrm{GL}_{n,K}$ be a faithful $K$-representation. Then the following two conditions are equivalent.
	\begin{enumerate}
	    \item[(a)] The algebraic group $G$ is non-reductive, or the derived group $\Gamma^{\mathrm{der}}$ of $\Gamma$ is infinite.
	    \item[(b)] There exists $\delta>0$ such that $\#\{\gamma \in \Gamma \colon H_{\mathrm{aff}}(\rho(\gamma))\leq T\}\gg T^{\delta}$ as $T\to \infty$.
	\end{enumerate}
		
	\end{theo}

    We will prove this theorem in \S \ref{vol}. In view of Theorem \ref{firstmainthm}, this theorem implies that for an $S$-arithmetic group $\Gamma$ of a linear algebraic $K$-group $G$, if either $\Gamma$ is non-virtually abelian or $G$ is $K$-isotropic, then all of its (PEP) subset are {\bf sparse} in it (see also Theorem \ref{maingpth}).

    Theorem \ref{firstmainthm} is proved by comparing $h_{\mathrm{aff}}(\mathbf{f}(\mathbf{n}))$ and $\|\mathbf{n}\|_{\infty}$. To elucidate the idea, we now state one representative result that will be sharpened and made more precise, in the form of asymptotic formulas, in \S \ref{pfht} (cf. Theorem \ref{mainlongpaper} and Corollary \ref{CoromainThm}).

	\begin{theo}
	\label{minspecial} Let $\mathbf{f}$ be a vector of purely exponential polynomials in $r$ variables. Then there exists an explicitly computable constant $C=C(\mathbf{f})>0$ such that for all but finitely many vectors $\mathbf{b}\in \mathbf{f}(\ZZ^r)$, there exists an $\mathbf{n}\in \mathbf{f}^{-1}(\mathbf{b})$ such that
		\begin{equation}\label{primaryht}
		h_{\mathrm{aff}}({\mathbf{f}}({\mathbf{n}}))\geq C\cdot \|\mathbf{n}\|_{\infty}.
		\end{equation}
	\end{theo}
 It should be emphasized that while the constant $C$ in (\ref{primaryht}) is explicitly computable, the number of exceptional vectors seems to be ineffective in general. More precisely, Theorem \ref{minspecial} will be deduced from Corollary \ref{CoromainThm}, where the number of exceptions (exceptional cosets, to be precise) cannot be explicitly bounded. We will discuss more effectivity issue in \S \ref{noteffective}.
	
\begin{rema} {\rm It is plain that for any vector of purely exponential polynomials $\mathbf{f}$, there also exists an effective $M>0$ such that 
	$$h_{\mathrm{aff}}({\mathbf{f}}({\mathbf{n}}))\leq M\cdot (\|\mathbf{n}\|_{\infty} +1)$$
	for all $\mathbf{n}=(n_1,\dots, n_r)\in \mathbb{Z}^r$, which implies that inequality (\ref{primaryht}) is in a sense best possible. We also observe that we can eliminate the exceptional set in Theorem \ref{firstmainthm} by replacing $C$ by another constant $C'$ which is no longer explicitly computable. In other words, we can show that there exists an ineffective constant $C'=C'(\mathbf{f})>0$ such that for all $\mathbf{b}\in \mathbf{f}(\ZZ^r)$, there exists $\mathbf{n}\in \mathbf{f}^{-1}(\mathbf{b})$ such that
	$$h_{\mathrm{aff}}({\mathbf{f}}({\mathbf{n}}))+1\geq C'\cdot \|\mathbf{n}\|_{\infty}.$$
	}
\end{rema}
	
	\begin{rema} \label{minnece} {\rm 
		It should be pointed out that (\ref{primaryht}) may fail to hold for {\bf all} vectors $\mathbf{n}\in \mathbf{f}^{-1}(\mathbf{b})$. For example, let $g_1,\dots, g_r\in \mathrm{SL}_n(K)$ be semi-simple matrices such that $g_1^n\cdots g_r^n=1$ for all $n\in \mathbb{Z}$ (this happens, for instance, when the $g_i$'s pairwise commute and their product is the identity matrix) and consider the (PEP) set $\langle g_1 \rangle \cdots \langle g_r \rangle$ (which is equal to the subgroup generated by $g_1,\dots, g_r$ if they all commute). Then the height of $\mathbf{f}(n,n,\dots,n)=g_1^n\cdots g_r^n=1$ is constant, whereas the norm of $(n,n,\dots,n)\in \mathbf{f}^{-1}(1)$ is $n$, which goes to infinity as $n\to \infty$. In fact, the example above has significance beyond its own right. We will return to this example and discuss it in more detail in Section \ref{Par.Open}.  }
	\end{rema}
	
	Another consequence of Theorem \ref{minspecial} is the following, which we formulate as a lemma as it will be a crucial ingredient in the proof of Theorem \ref{secondmainthm} below.
	\begin{lemma}\label{Cor}
		Let $K$ be a number field, $\Sigma \subseteq \mathrm{GL}_n(K)$ be a (PEP) subset in $r$ variables, and let $g\in \mathrm{GL}_n(K)$ be a non-semi-simple matrix.
		Then $\# \{l\in \mathbb{Z}\colon |l|\leq N\text{ and }g^l\in \Sigma\}
		=O((\log N)^{r+1})$ as $N\to \infty$.
	\end{lemma}

    We obtain, in particular:
 
    \begin{coro} Let $g\in \mathrm{GL}_n(K)$ be a non-semi-simple matrix, and let $g_1,\dots,g_r \in \mathrm{GL}_n(K)$ be semi-simple. Then the set 
		$$\{ l\in \mathbb{Z}\colon g^l\in  \langle g_1 \rangle \cdots \langle g_r \rangle, |l|<N  \}$$
		has size $O((\log N)^{r+1})$ as $N\to \infty$.
	\end{coro}
	
	The above quantitative Theorem \ref{firstmainthm} enables one to obtain the following key group-theoretic result. This result, originally announced in \cite{CDRRZcr}, provides a {\it criterion} for when a linear group admits (BG) by semi-simple elements. In this sense, the following theorem considerably sharpens the main theorem of \cite{CRRZ}, in which there was only a necessary condition.
     
	\begin{theo}[Main application in group theory]
	\label{secondmainthm} 
		Let $K$ be a field of characteristic zero and let $\Gamma \subset \mathrm{GL}_n(K)$ be a linear group. Then the following three properties are equivalent. 
		\begin{enumerate}
			\item $\Gamma$ admits a (PEP) parametrization.
			\item $\Gamma$ is anisotropic and has (BG).
			\item $\Gamma$ is finitely generated and the connected component $G^{\circ}$ of the Zariski closure $G$ of $\Gamma$ is a torus (in particular, $\Gamma$ is virtually abelian). 
		\end{enumerate}
	\end{theo}

 Here we say that a linear group $\Gamma \subset \GL_n(K)$ over a field of characteristic zero is {\bf anisotropic} if it consists of semi-simple elements. This terminology is inspired by the fact that for a semi-simple algebraic group $G$ over a field $K$ of characteristic zero, $G$ is $K$-anisotropic (i.e.\ it admits no non-trivial cocharacter $\mathbb{G}_m \hookrightarrow G$ \cite[p.~65]{PR}) if and only if $G(K)$ contains only semi-simple elements.

    To prove Theorem \ref{secondmainthm} we first reduce to the number field case via a specialization argument and then we deal with the number field case by using Theorem \ref{firstmainthm}. See Section \ref{Sec.MainProofs} for details.

 It is worth mentioning that Theorem \ref{secondmainthm} gives some surprising implications whose verification seems unrealistic by using elementary methods. For example, (1) if a linear group $\Gamma$ admits (PEP) parametrization, then it must be \underline{finitely generated}, (2) linear groups having (BG) by semi-simple elements \underline{exhaust} all linear groups with (PEP) parametrization (caution: such groups still may have (PEP) parametrizations other than (BG), cf. Remark \ref{gpimagenopep}).

	Moreover, Theorem \ref{secondmainthm} implies that over a field of characteristic zero, linear groups that are boundedly generated by semi-simple elements are precisely those groups which are virtually abelian and anisotropic. This actually provides a {\bf criterion} for a linear group to be boundedly generated by semi-simples, and thereby considerably sharpens the main result in \cite{CRRZ} asserting only
  that such groups are virtually solvable. This upgrade is achieved by using more sophisticated techniques from the theory of diophantine approximation.

    In the appendix we will give a (stronger) version of Theorem \ref{secondmainthm} in a more quantitative direction, namely Theorem \ref{maingpth}. 

	\vskip5mm
	\noindent
	{\bf Notations and Conventions}
	\begin{enumerate}
	\item[*] All fields in this paper have characteristic zero. When $K$ is a number field, we denote by $V^K, V^K_{\infty},V^K_f$ the set of all/archimedean/non-archimedean places of $K$, respectively. And $S$ is always a finite subset of $V^K$ containing $V^K_{\infty}$.
	
	\item[*] We introduce several compact notations to represent a purely exponential polynomial. Let $K$ be a field of characteristic $0$. For a vector $\mathbf{b}=(b_1,\dots,b_k)\in (K^*)^k$, $\mathbf{n}=(n_1,\dots, n_r)\in \ZZ^r$ and a matrix $A=(a_{ij})\in M_{k\times r}(\mathbb{Z})$, we use the notation $\mathbf{b} ^{A\cdot \mathbf{n}^T}$ as a contraction for the following expression:
	$$b_1^{a_{11}n_1+\cdots +a_{1r}n_r}b_2^{a_{21}n_1+\cdots +a_{2r}n_r}\cdots b_k^{a_{k1}n_1+\cdots +a_{kr}n_r}$$
	(note that the linear forms appearing on the exponents above are the ones associated to the rows of $A$). 
	If we let $\mathbf{r}_i\coloneqq(a_{i1},a_{i2},\dots,a_{ir})$ be the $i$-th row of $A$ and we write $b^{\mathbf{r}_i\cdot \mathbf{n}^T}$ for $b^{a_{i1}n_1+\cdots +a_{ir}n_r}$, then the above expression is also denotable by:
	$$\mathbf{b} ^{A\cdot \mathbf{n}^T}=b_1^{\mathbf{r}_1\cdot \mathbf{n}^T}b_2^{\mathbf{r}_2\cdot \mathbf{n}^T}\cdots b_k^{\mathbf{r}_k\cdot \mathbf{n}^T}.$$
	
	Clearly every purely exponential polynomial $f\colon \ZZ^r\to \overline{K}$ can be written as a formal linear combination of terms of shape $\u{\lambda} ^{A\cdot \mathbf{n}^T}, \text{ where }\u{\lambda}=(\lambda_1,\ldots,\lambda_k)\in \overline{K}^k, A=\begin{pmatrix}
	\mathbf{r}_1 \\ \vdots \\ \mathbf{r}_k
	\end{pmatrix}\in M_{k \times r}(\mathbb{Z})$. 
	\item[*] For every purely exponential polynomial $f\colon \ZZ^r\to \overline{K}$ in this paper, we fix an expression $f(\mathbf{n})=a_1\boldsymbol{\lambda}^{A_1\cdot \mathbf{n}^T}+\cdots+a_k\boldsymbol{\lambda}^{A_k\cdot \mathbf{n}^T}$ as in \eqref{Eq:expressionPEP} such that $a_i\neq 0$ and $\boldsymbol{\lambda}^{A_1\cdot \mathbf{n}^T},\dots, \boldsymbol{\lambda}^{A_k\cdot \mathbf{n}^T}$ are distinct characters of $\ZZ^r$, which we call the {\em characters} {\bf appearing} in $f$. We also assume     
    the following natural conditions: (1) for each component $f_i$, the number $e$ of terms in it is {\bf minimal}, in particular $(l_{1,j}, \ldots l_{k,j})\neq (l_{1,j'}, \ldots l_{k,j'})$ if $j \neq j'$; (2) for each basis $\lambda_{\alpha}$ and integers $b_t \ (1\leq t\leq k,t\neq \alpha)$, we have $\lambda_{\alpha}\neq \prod\limits_{t\neq \alpha}\lambda_t^{b_t}$ (in particular, the bases $\lambda_1,\dots, \lambda_k$ are distinct, and none of them is equal to $1$); (3) the number of bases is also minimal: in other words, for each basis $\lambda_{\alpha},1\leq \alpha \leq k$, there is a component $f_i$ which has a term such that its exponent $l(\mathbf{x})$ of $\lambda_{\alpha}$ is such that $\lambda_{\alpha}^{l(\mathbf{x})}$ is non-constant on $\ZZ^r$. 
	\item[*] Whenever we say {\it distribution} or {\it distribution rate} of points in a subset $\Sigma \subset K^n$ ($K$ is a number field), we mean the asymptotic behavior of the function $N(\Sigma,T)\coloneqq\#\{P\in \Sigma\colon H_{\mathrm{aff}}(P)<T\}$ as $T \to \infty$.

	    \item[*] For an abstract group $\Gamma$, we say that $\Gamma$ {\it virtually} has a property $\mathcal{P}$ if $\Gamma$ has a finite index subgroup which has property $\mathcal{P}$.
	   
	    \item[*] The acronym (PEP) will be used both as a noun (purely exponential parametrization) and as an adjective (e.g.\  purely exponentially parametrized).
	   
	    \item[*] The {\it degeneracy type} of a sum $s_1+\cdots +s_r$ (the $s_i$'s are in a field of characteristic zero) is a partition (not necessarily unique for a given sum, but it always exists and there are only finitely many possibilities, so we fix one) $\{1,2,\dots,r\}=T_1 \sqcup T_2$ (with $T_1$ and $T_2$ possibly empty) such that $\sum_{i \in T_1} s_i$ is non-degenerate and $\sum_{i \in T_2} s_i=0$.
            \item[*] A $K$-group $G$ is said to be {\em anisotropic} if it has no non-trivial $K$-split subtori, or equivalently if there does not exist an embedding $\mathbb{G}_m \hookrightarrow G$. (See e.g.\ \ \cite[p.~65]{PR}). When $G$ is semi-simple, this is equivalent to $G(K)$ containing only semi-simple elements.
	\end{enumerate}
	
	\section{Brief review of heights}\label{heightdefprop}
	
	In order to describe the sparseness of (PEP) sets, we count points of linear groups over number fields in terms of the height function (via the usual projective height in Diophantine Geometry). Here we only define the height function and state the necessary properties without proofs. For a detailed account of height functions, see standard literature in Diophantine Geometry including \cite[\S B]{HindrySilverman} or \cite{BombieriGubler}.
	
	Let $K$ be a number field. For a valuation $v$ of $K$, let $\| \cdot \|_v$ be the corresponding normalized absolute value such that the usual product formula holds, i.e.\  $\prod\limits_{v\in V^K}\| x \|_v=1$ for all $x\in K^{*}$. 
	\begin{defi} {\rm Let $P=(x_0\colon x_1\colon \cdots\colon x_n)\in \mathbb{P}^n(K)$ be a point whose homogeneous coordinates are chosen in $K$. Then its {\bf absolute (multiplicative) height} is defined by 
		$$H(P)\coloneqq\Big( \prod_{v\leq \infty}\max\{ \| x_0 \|_v,\| x_1 \|_v,\dots, \| x_n \|_v  \}\Big)^{1\slash [K\colon \mathbb{Q}]}.$$
		The {\bf absolute logarithmic height} of $P$ is given by $h(P)=\log H(P)$. }
	\end{defi}
	\begin{rema} {\rm The product formula guarantees that both heights $H(P)$ and $h(P)$ are independent of the choice of homogeneous coordinates of the point $P$. The exponent $1\slash [K\colon \mathbb{Q}]$ above enables one to define heights on $\mathbb{P}^n(\overline{\mathbb{Q}})$ unambiguously. }
	\end{rema}
	\begin{defi}[Affine Height] \label{defheight}{\rm For a vector $\mathbf{x}=(x_1,\dots, x_n)\in K^n$, we denote by 
		$$H_{\mathrm{aff}}(\mathbf{x})\coloneqq H(1\colon x_1\colon \cdots\colon x_n)$$
		the {\bf affine (multiplicative) height} of the vector. We define the corresponding {\bf affine logarithmic height} as $h_{\mathrm{aff}}(\cdot)\coloneqq\log H_{\mathrm{aff}}(\cdot)$.}
	\end{defi}
	Similarly to $H(\cdot)$ and $h(\cdot)$, our $H_{\mathrm{aff}}(\cdot)$ and $h_{\mathrm{aff}}$ also extend unambiguously to $\overline{\QQ}^n$.
	The following finiteness lemma (see \cite[Theorem B.2.3]{HindrySilverman} for its proof) shows that the height functions in Definition \ref{defheight} are appropriate ways to count $K$-points in affine spaces, in particular in $\mathrm{GL}_n(K)\subset \mathbb{A}_K^{n^2}$.
	\begin{lemma}[Northcott's theorem] For any $t\in \mathbb{R}, T\in \mathbb{R}_{>0}, d \in \NN$, the sets 
		$$\{\mathbf{x}\in \overline{\QQ}^n\colon H_{\mathrm{aff}}(\mathbf{x})\leq T, \deg \mathbf{x} \leq d\}\text{    and    }\{\mathbf{x}\in \overline{\QQ}^n\colon h_{\mathrm{aff}}(\mathbf{x})\leq t, \deg \mathbf{x} \leq d\},$$
		where $\deg (x_1,\ldots,x_n):=[\QQ(x_1,\dots,x_n)\colon \QQ]$, are both finite. 
	\end{lemma}

    In particular, for a fixed number field $K$, the set $\{\mathbf{x}\in K^n\colon H_{\mathrm{aff}}(\mathbf{x})\leq T\}$ (and the analog with the logarithmic height) is finite.
 
	We end by recalling three useful and straightforward lemmas. 
 
	\begin{lemma}\label{Lem:heightdifference} 
		For $\mathbf{x}=(x_0:x_1:\dots:x_r),\mathbf{y}=(y_0:y_1:\dots:y_r)\in (K^{*})^{n+1}/K^*$, define $\mathbf{x}\cdot \mathbf{y}:=(x_0y_0:\dots:x_ry_r)$ and $\mathbf{x}^{-1}:=(x_0^{-1}:\dots:x_r^{-1})$. Then $H(\mathbf{x}^{-1})^{-1} \cdot H(\mathbf{y}) \leq H(\mathbf{x}\cdot \mathbf{y})\leq H(\mathbf{x}) \cdot H(\mathbf{y})$.
	\end{lemma}

	\begin{lemma}\label{htvector} For $\mathbf{x}=(x_0,x_1,\dots,x_r)$, we have $H_{\mathrm{aff}}(\mathbf{x})\leq H_{\mathrm{aff}}(x_0)\cdots H_{\mathrm{aff}}(x_r)$.
	\end{lemma}

	Notice that $H_{\mathrm{aff}}(t)=H_{\mathrm{aff}}(\frac{1}{t}), \ t \in K^*.$ Indeed, this follows by the identity $(1\colon t)=(1\slash t \colon 1)\in\mathbb{P}^1(K)$.
	
	\begin{lemma}\label{lfxbd} For $x_1,\dots,x_r\in K$, we have
		$$\frac{h_{\mathrm{aff}}(x_1)+\cdots +h_{\mathrm{aff}}(x_r)}{r}\leq h_{\mathrm{aff}}(x_1,\dots,x_r)\leq h_{\mathrm{aff}}(x_1)+\cdots + h_{\mathrm{aff}}(x_r).$$
	\end{lemma}
	
	\section{Key Diophantine ingredients}\label{dioarguments}
	In this section we provide some necessary Diophantine-approximation inputs, mostly related to the subspace theorem.

	We recall that a sum
	\[
	x_1+\cdots+x_r
	\]
	is called {\bf non-degenerate} 
 if no non-trivial subsum (including the whole sum) is equal to $0$, and {\bf degenerate} otherwise.

\vskip1mm

    \begin{defi}\label{defdegenerate} {\rm For a vector of purely exponential polynomials $\mathbf{f}=(f_1,\dots, f_s)$, an integer vector $\mathbf{n}\in \ZZ^r$ is called $\mathbf{f}$-non-degenerate, or we say that $\mathbf{f}$ is non-degenerate at $\mathbf{n}$ if each component $f_i(\mathbf{n})$, as a linear sum of distinct characters of $\ZZ^r$ evaluated at $\mathbf{n}$, is non-degenerate, i.e.\  no subsum vanishes. The vector $\mathbf{n}$ is called $\mathbf{f}$-degenerate otherwise. When there is no risk of confusion, we also simply say that $\mathbf{f}(\mathbf{n})$ is degenerate/non-degenerate, or $\mathbf{n}$ is degenerate/non-degenerate.}
	\end{defi}
 Many of the discussions in this paper need the following definition, which plays a central role for our method. This definition introduces a simple type of vectors of purely exponential polynomials which one can usually reduce to, see the end of \S \ref{generalitypep}.

     \begin{defi}{\rm Let $\mathbf{f}=(f_1,\dots,f_s)$ be a vector of purely exponential polynomials in $r$ variables. Then we say that $\mathbf{f}$ is {\bf reduced} or {\bf in reduced form}, if its bases $\lambda_1,\dots, \lambda_k$ are multiplicatively independent, i.e.\  $\lambda_1^{\alpha_1}\cdots \lambda_k ^{\alpha_k}=1 (\alpha_i\in \ZZ)\Leftrightarrow \alpha_1=\cdots =\alpha_k=0$, and its {\em exponents} (that is, the linear forms appearing in the exponents of the monomials of all components of $\mathbf{f}$ as in \eqref{Eq:expressionPEP}) span over $\mathbb{Q}$ the dual space of $\mathbb{Q}^r$.}
	\end{defi}

    The most representative result of this section is the following asymptotic property.
    \begin{prop}\label{Prop:norm}
        Let $\mathbf{f}=(f_1,\ldots, f_s)\colon \ZZ^r \to K^s$ be a vector of purely exponential polynomials, and let $\boldsymbol{\lambda}^{A_i\cdot \mathbf{n}^T}, i=1,\ldots,t$ be all the characters appearing in the various $f_i$'s. Then for non-degenerate $\u{n} \in \ZZ^r$, we have
      \[
      h_{\mathrm{aff}}(\mathbf{f}(\u{n}))\sim \|\u{n}\|_{\mathbf{f}} \coloneqq h_{\mathrm{aff}}\left(\boldsymbol{\lambda}^{A_1\cdot \mathbf{n}^T},\ldots, \boldsymbol{\lambda}^{A_t\cdot \mathbf{n}^T}\right),
      \]
     and $\|\cdot \|_{\mathbf{f}}\colon \ZZ^r\to \RR_{\geq 0}$ is a semi-norm induced by a norm on $(\ZZ^r/\mathcal{K}) \otimes_{\ZZ} \R$, where $\mathcal{K}\coloneqq \bigcap\limits_{i=1}^t \ker \chi_i$. 
     If $\mathbf{f}$ is reduced, then $\|\cdot \|_{\mathbf{f}}$ is a norm.
    \end{prop}

    The error term, i.e.\ the rate of convergence implicit in the $\sim$ sign above is unavailable with our method.
    
    \vskip2mm

    The proof of Proposition \ref{Prop:norm} relies on a so-called ``semi-effective'' result, which is a variation of a height inequality of Evertse \cite{evertse84}, whose proof follows from the version of the subspace theorem due to Schlickewei-Schmidt cf. \cite[Theorem 2.3]{CorvajaZannier}. 

    To formulate our version of Evertse's inequality, we need to adopt some of his notation with mild modification.
    For a vector $\mathbf{x}=(x_0,x_1,\dots,x_r)\in \cO_S^{r+1}$, define the $S$-height and $S$-norm as follows:
	$$H_S(\mathbf{x})=H_S(x_0,x_1,\dots,x_r)\coloneqq\left(\prod_{v\in S} \max _i \|x_i\|_v\right)^{1\slash [K\colon \mathbb{Q}]},\text{    }N_S(x_0x_1\cdots x_r)\coloneqq\prod_{v\in S} \|x_0x_1\cdots x_r\|_v.$$
	Notice that the $S$-height here differs from the one in \cite[\S 6]{EG15} by a normalizing exponent. Also, it is evident that when $x_0,\dots, x_r$ are all $S$-units, we have $N_S(x_0x_1\cdots x_r)=1$ and $H_S(\mathbf{x})=H(\mathbf{x})$.

 The original version of Evertse's theorem cf. \cite[Proposition 6.2.1]{EG15} reads 
		\vskip2mm
		\noindent
		{\bf Theorem.} (Evertse) Let $K$ be a number field, $S$ be a finite set of places containing all archimedean ones. Let $T$ be a subset of $S$. Then for any $\varepsilon>0$ there exists a(n ineffective) constant $C>0$ such that for any non-degenerate $\mathbf{x}=(x_1,\ldots, x_n) \in \cO_S^n$, we have:
		$$N_S(x_1\cdots x_n)\prod_{v\in T}\|x_1+\cdots+x_n\|_v \geq C \cdot H_T(\mathbf{x})^{[K\colon \mathbb{Q}]}\cdot H_S(\mathbf{x})^{-\varepsilon}.$$
		(Note that the exponent on $H_T$ does not appear in \cite{EG15} because our normalizations are different.)

  In particular, taking $T=\{v\}$ for any $v\in S$ and letting $x_i\in \cO_S^*$ for all $i$, we obtain our variation of Evertse's inequality as follows (see also \cite[Theorem 4]{CZcompositio02} for a result in a similar flavor, beware that the notion of $S$-height there differs from ours).

	\begin{theo}\label{Lemma1}
		Fix a natural number $n$, a number field $K$ and a finite set of places $S$ of $K$ containing all archimedean ones. Then for any $\varepsilon>0$, there exists a(n ineffective) constant $C>0$ such that for all non-degenerate sums $x_1+\cdots+x_n$ with $\mathbf{x}=(x_1,\ldots, x_n) \in (\cO_S^*)^n$ and all $v \in S$:
		\begin{equation}\label{Eq:Lemma1}
		\|x_1+\cdots+x_n\|_v \geq C \cdot \max\{\|x_1\|_v,\ldots,\|x_n\|_v\} \cdot H_S(\mathbf{x})^{-\varepsilon}.
		\end{equation}
	\end{theo}

	\begin{coro}\label{Lemma2}
		Let $K,S$ be as above and $n_1,\ldots,n_k\in \NN$. Then for every $\varepsilon >0$ there exists a(n ineffective) constant $C>0$ such that for all $k$-tuples of non-degenerate sums $y_i=x_{i1}+\cdots+ x_{in_i}, \ i=1,\ldots,k$ with $x_{ij} \in \cO_S^*$:
		\begin{equation}\label{EqLemma2}
		C \cdot H_{\mathrm{aff}}(\mathbf{x})^{1-\varepsilon} \leq H_{\mathrm{aff}}(y_1, \ldots, y_k) \leq  C_0 \cdot H_{\mathrm{aff}}(\mathbf{x}),
		\end{equation}
		where $\mathbf{x}$ denotes the vector $(x_{11},\ldots,x_{1n_1}, \ldots, x_{k1},\ldots, x_{kn_k}) \in (\cO_S^*)^{n_1+\cdots+n_k}$ and $C_0=\max_i{n_i}$.
	\end{coro}
	\begin{proof}
		The inequality on the right follows immediately from the triangle inequality, so we prove the one on the left. We let $\mathbf{x}_i\coloneqq(x_{i1},\ldots,x_{in_i})$ for each $i$. 
		By Theorem \ref{Lemma1}, we may take $C>0$ such that $\|y_i\|_v \geq C \cdot \max\{\|x_{i1}\|_v,\ldots,\|x_{in_i}\|_v\} \cdot H_S(\mathbf{x}_i)^{-\varepsilon}$ for each $i$ and each $v \in S$. Moreover, we may assume that $C \leq 1$.
		Since $H_S(\mathbf{x}_i)=H_{\mathrm{aff}}(\mathbf{x}_i) \leq H_{\mathrm{aff}}(\mathbf{x})$ for each $i$, we deduce from \eqref{Eq:Lemma1} that
\begin{align*}
    H_{\mathrm{aff}}(y_1,\dots, y_k)^{[K\colon \mathbb{Q}]}&\geq \prod_{v \in S}\max_i \{\|y_i\|_v,1\}\\
    & \geq \prod_{v \in S} \left( \max\left\{C \cdot H_{\mathrm{aff}}(\mathbf{x})^{-\varepsilon} \cdot\max_i\{\max\{\|x_{i1}\|_v,\ldots,\|x_{in_i}\|_v\}\} ,1\right\} \right)\\
    &\geq C^{\#S} \cdot H_{\mathrm{aff}}(\mathbf{x})^{-\#S \cdot \varepsilon} \cdot \prod_{v \in S} \max_i\{\max\{\|x_{i1}\|_v,\ldots,\|x_{in_i}\|_v\},1\} \\
    & \geq C^{\#S}\cdot H_{\mathrm{aff}}(\mathbf{x})^{[K\colon \mathbb{Q}]-\#S\cdot \varepsilon}.
\end{align*}

		After raising the above inequality to the power of $1\slash [K\colon \mathbb{Q}]$, and renaming $C$ and $\varepsilon$, we conclude the proof.
	\end{proof}
	
	As a consequence we get the following ``anti-triangular'' inequality.
	
	\begin{coro}\label{evertseinequality}
		Let $S$ be a finite set of places of a number field $K$ containing all archimedean ones. Then for every $\varepsilon >0$, the inequality
		\[
		h_{\mathrm{aff}}(s_1+\cdots+s_r) < \left( \frac{1}{r}-\varepsilon \right) \cdot (h_{\mathrm{aff}}(s_1)+\cdots+h_{\mathrm{aff}}(s_r))\text{   with   }s_i\in \cO_S^{*}
		\]
		has only finitely many solutions such that the sum $s_1+\cdots+s_r$ is non-degenerate (the number of solutions is {\bf ineffective}).
	\end{coro}

	\begin{proof}
		This follows immediately by letting $k=1$ in Corollary \ref{Lemma2}, taking the logarithm of the inequality on the right in \eqref{EqLemma2} and using Lemma \ref{lfxbd}.
	\end{proof}
	
	Corollary \ref{Lemma2} yields another crucial consequence.
	
	\begin{coro}\label{Lem:exactasymp}
		Let $\mathbf{f}\colon \mathbb{Z}^r \to K^n$ be a vector of purely exponential polynomials, and let $u_i(\u{n}):=a_i \cdot \boldsymbol{\lambda}^{A_i\cdot \mathbf{n}^T}, i=1,\ldots,t$ be the list of all monomials of its components (the order or any possible repetitions are irrelevant). Then there exists a(n explicitly computable) constant $c_0$ and also, for every $\varepsilon >0$, there exists a(n ineffective) $c \in \R$ such that for all non-degenerate $\mathbf{n}$ we have:
		\[
		 (1-\varepsilon) h_{\mathrm{aff}}\left(\boldsymbol{\lambda}^{A_1\cdot \mathbf{n}^T},\ldots, \boldsymbol{\lambda}^{A_t\cdot \mathbf{n}^T}\right)-c  \leq h_{\mathrm{aff}}(\mathbf{f}(\mathbf{n})) \leq  h_{\mathrm{aff}}\left(\boldsymbol{\lambda}^{A_1\cdot \mathbf{n}^T},\ldots, \boldsymbol{\lambda}^{A_t\cdot \mathbf{n}^T}\right) +c_0 .
		\]
	\end{coro}
	\begin{proof}
		Extending $K$ and choosing $S$ appropriately, we assume that all bases and coefficients of the $u_i$ are $S$-units in $K$. 
		Applying Corollary \ref{Lemma2} to the sums corresponding to the components of $\mathbf{f}$, we deduce that there exist $c_0,c \in \R$ such that:
		\[
	(1-\varepsilon) h_{\mathrm{aff}}\left(u_1(\mathbf{n}),\ldots, u_k(\mathbf{n})\right)-c	   \leq h_{\mathrm{aff}}(\mathbf{f}(\mathbf{n})  \leq h_{\mathrm{aff}}(u_1(\mathbf{n}),\ldots, u_k(\mathbf{n})) +c_0 
		\]
      
		To conclude it suffices to observe that $h_{\mathrm{aff}}(u_1(\mathbf{n}),\ldots, u_k(\mathbf{n}))- h_{\mathrm{aff}}\left(\boldsymbol{\lambda}^{A_1\cdot \mathbf{n}^T},\ldots, \boldsymbol{\lambda}^{A_k\cdot \mathbf{n}^T}\right)=O(1)$ as $\|\mathbf{n}\|_{\infty}\to \infty$ by Lemma \ref{Lem:heightdifference}, where the bounding constant is explicitly computable in terms of $a_1,\ldots,a_k$.
	\end{proof}
	Lastly, we need the following crucial fact (which we will repeatedly use). Despite the simplicity of its statement, the verification is non-trivial.
	
	\begin{prop}\label{Prop:norm2} 
        Keeping the notation and assumptions of the above corollary, let $\mathcal{K}\coloneqq \bigcap\limits_{i=1}^t \ker (\mathbf{n}\mapsto \boldsymbol{\lambda}^{A_i\cdot \mathbf{n}^T})$. 
		 Then the function $\u{n} \mapsto h_{\mathrm{aff}}\left(\boldsymbol{\lambda}^{A_1\cdot \mathbf{n}^T},\ldots, \boldsymbol{\lambda}^{A_t\cdot \mathbf{n}^T}\right)$ extends to a norm on $(\ZZ^r/\mathcal{K})\otimes_{\ZZ}\RR$.
	\end{prop}
      This proposition can be verified by using Minkowski's first theorem about convex bodies, see \cite[p.\ 136]{BombieriGubler} for details. It is worth pointing out that the method in \cite{BombieriGubler} is a natural analog to Cassels' proof \cite[Appendix D]{Cassels1966} of the fact that the canonical height $\hat{h}$ of an elliptic curve $E$ over a number field $K$ can be extended to a positive definite quadratic form on $E(K)\otimes_{\ZZ}\RR$. See also \cite[\S 3.4]{Zannier09} for a self-contained proof of Proposition \ref{Prop:norm2}. 
\begin{proof}[Proof of Proposition \ref{Prop:norm}]
    Combine Proposition \ref{Prop:norm2} with Corollary \ref{Lem:exactasymp}, and note that when $\mathbf{f}$ is reduced, $\mathcal{K}$ is trivial.
\end{proof}
	
	\section{Generalities about Purely Exponential Polynomials}\label{generalitypep}
	In this section, we provide some further examples of (PEP), and prove several fundamental properties of PEP sets. Other than as subsets of a linear group possessing bounded generation by semi-simple elements, (PEP) sets also arise naturally in the following situations. 
	
	\begin{exam} {\rm Let $S=\{v_1,\dots,v_s\}$ be a finite set of places of a number field $K$, containing all infinite ones. Then according to Dirichlet's $S$-unit theorem, there exist $\varepsilon_1,\dots,\varepsilon_{s-1}$ such that every $\varepsilon \in \cO_S^*$ can be expressed uniquely as 
		$$\varepsilon=\eta \varepsilon_1^{b_1}\dots \varepsilon_{s-1}^{b_{s-1}}$$
		where $\eta$ is a root of unity in $K$ and $b_1,\dots,b_{s-1}$ are rational integers. (Recall that such $\{\varepsilon_1,\dots,\varepsilon_{s-1}\}$ is called a {\bf fundamental system of $S$-units}.) Thus $\cO_S^*$ is a (PEP) set. This is relevant to the discussion in \S \ref{noteffective}. }
	\end{exam}
	More classical but equally important examples of (PEP) include linear recurrence sequences.
	\begin{exam} {\rm Any simple linear recurrence sequence (those whose characteristic polynomial has no multiple roots) produces a (PEP) set. A lot of interesting diophantine properties of this kind of (PEP) sets have been discovered by Corvaja, Evertse, van der Poorten, Schlickewei, Schmidt, and Zannier etc. See \cite{evertse84,Schmidt99,CZcompositio02}, the survey article of Schmidt \cite{Schmidt00}, or the book by Corvaja-Zannier \cite[Chapter 4]{CorvajaZannier} for a survey of recent developments in this direction. }
	\end{exam}
	
	The following assertion shows that a finite union of (PEP) sets is still a (PEP) set. 
	\begin{prop}\label{Prop:Union} 
		Let $\Sigma_1,\ldots, \Sigma_s$ be (PEP) sets in $K^r$; then 
		$$\Delta =\Sigma_1\cup\cdots \cup\Sigma_s $$
		is also a (PEP) set.
	\end{prop}
	\begin{proof} It suffices to consider the case $s=2$ and then conclude by induction. Let $\mathbf{f_i}\colon \mathbb{Z}^{m_i}\to K^r,i=1,2$ be vectors of purely exponential polynomials such that $\mathbf{f_i}(\mathbb{Z}^{m_i})=\Sigma_i.$
		Consider $\widetilde{\mathbf{f}}\colon \mathbb{Z}^{m_1+m_2+1}\to K^r$ defined as
		$$\widetilde{\mathbf{f}}(x_1,\dots,x_{m_1},y_1,\dots, y_{m_2};z)=\frac{\mathbf{f}_1(\mathbf{x})+\mathbf{f}_2(\mathbf{y})}{2}+(-1)^z\cdot \frac{\mathbf{f}_1(\mathbf{x})-\mathbf{f}_2(\mathbf{y})}{2}.$$
		This is also a (PEP), and its image is exactly $\Sigma_1\cup\Sigma_2$.
	\end{proof}
	As a consequence, we obtain that any {\bf finite set} in $K^r$ is a (PEP) set, which is false in general for ordinary polynomially parameterizable sets. 
	\begin{rema} {\rm Although the above proposition guarantees that a (PEP) set is always the image of a single vector of purely exponential polynomials, in practice, it is often helpful to do the converse, i.e.\  we might need to decompose a (PEP) set into simpler ones, see Proposition \ref{decomposepep}.}
	\end{rema}
	
	\begin{defi}{\rm
		The {\bf rank} of a coset $M$ in a commutative group is the rank of the abelian group $M-M$, which we still denote by $\rk M$, or more precisely, $\rk _{\ZZ} M$. By convention, $\rk_{\ZZ} \varnothing=-1$.}
	\end{defi}
	With the help of this notion, we introduce the following useful terminology.
	\begin{defi} \label{quasicoset} A {\bf quasi-coset} in $\ZZ ^r$ is a subset of $\ZZ^r$ which has shape $M\backslash (N_1\cup \cdots \cup N_l)$ where $M$ is a coset of $\ZZ^r$ and $N_i$ are cosets contained in $M$ such that $\rk_{\ZZ}N_i<\rk_{\ZZ} M$ for each $i$.
	
	\end{defi}
	The next result describes the zero set of a vector of purely exponential polynomials, and is basically a consequence of Laurent's theorem \cite[Theorem 2.7]{CorvajaZannier} (although we will adopt Evertse's result to prove it). Notice that the zeros of purely exponential polynomials, or more generally, of = exponential polynomials, have been studied in depth from different aspects, cf. e.g.\  \cite{SchlickeweiSchmidt00,Schmidt99,CSZ09}.
	\begin{theo}[Zero Locus Theorem] \label{zerolocuspep}
		Let $\mathbf{f}=(f_1,\dots,f_s)$ be a vector of purely exponential polynomials in $r$ variables; then the zero set $\{\mathbf{n}\in \ZZ^r\colon \mathbf{f}(\mathbf{n})=\mathbf{0}\}$ is a finite union of cosets, whose underlying groups belong to an explicitly computable finite set. Moreover, if the bases of $\mathbf{f}$ are multiplicatively independent, then those cosets all have rank $<r$. 
	\end{theo}
	In particular, letting $s=r=1$, Theorem \ref{zerolocuspep} implies that the zero locus of a simple linear recurrence is a union of a finite set with finitely many arithmetic progressions. This is also a well-known consequence of the Skolem-Mahler-Lech Theorem \cite[\S 2.1]{Recurrencebook}.
	\begin{rema} {\rm The condition ``multiplicatively independent'' in the last assertion of Theorem \ref{zerolocuspep} is necessary. For example the zero locus of $2^n+(-2)^n$ is $2\ZZ+1$, whose rank is equal to $1$, the rank of $\ZZ$.}
	\end{rema}
	In order to prove the Zero Locus Theorem, we introduce the following well-known property of $S$-unit equations, which turns out to be a non-trivial reformulation of Evertse-van der Poorten-Schlickewei's theorem about the finiteness of non-degenerate solutions of $S$-unit equations \cite[Theorem 2.4]{CorvajaZannier}.
	\begin{theo}\label{sunitfinitecoset}
		Let $K$ be a number field, $S$ be a finite set of places of $K$ containing all infinite ones. Let $h\in \NN$ and $a_1,\dots,a_h\in \cO^*_S$. Then the set of solutions of
		\begin{equation}\label{sunitequationcosets}
		a_1s_1+\cdots+a_hs_h=0,\text{ for }(s_1,\dots,s_h)\in (\cO_S^*)^h
		\end{equation}
		is {\bf equal} to a finite union of translates of subgroups of $(\cO_S^*)^h$. Moreover, there exists a finite set of explicitly computable subgroups of $(\cO_S^*)^h$ such that each of the translates above is a translate of one of these subgroups.
	\end{theo}
	
	\begin{proof} By the $S$-unit equation theorem, cf. \cite[Remark 2.5 (ii), p33]{CorvajaZannier}, the set of non-degenerate solutions of (\ref{sunitequationcosets}) is finite up to $\cO_S^*$-scaling, i.e.\ it is a union of finitely many translates of $\{(t,\dots,t)\colon t\in \cO_S^*\}\subset (\cO_S^*)^h$. On the other hand, for each degenerate solution $(s_1,\dots,s_h)$, there exists a partition $\{1,\dots,h\}=I_1\sqcup \cdots \sqcup I_l$ such that $\sum _{i\in I_j}a_is_i=0$ and the sum is non-degenerate for each $j=1,\dots,l$. (Notice: $|I_j|\geq 2$ for all $j$.) Using the $S$-unit equation theorem for these equations in fewer variables, we obtain that each has finitely many solutions up to $\cO_S^*$-scaling. We denote this finite set by $F_j\subset (\cO_S^*)^{I_j}\slash \cO_S^*$ and let $\widetilde{F_j}\subset (\cO_S^*)^{I_j}$ be a set of representatives of $F_j$. Then the set of non-degenerate solutions of the system
		$$\sum_{i\in I_j}a_is_i=0,\text{   for all }j=1,\dots,l$$
		is equal to $\prod_{j=1}^l  \widetilde{F_j}\cdot(\cO_S^*)$, which is itself a translate of an $\cO_S^*$-submodule of $(\cO_S^*)^h$.
		
		As for the second part of the statement, note that each $\prod_{j=1}^l  \widetilde{F_j}\cdot(\cO_S^*)$ above is a translate of the $\cO_S^*$-submodule $\prod_{j=1}^l  (\cO_S^*)_{\Delta_j}$, where $(\cO_S^*)_{\Delta_j}$ denotes the image of the diagonal embedding $\cO_S^*\hookrightarrow (\cO^*_S)^{I_j}$. Since there are only finitely many partitions of $\{1,\dots,h\}$, the theorem is proven.
	\end{proof}
	\begin{proof}[Proof of Theorem \ref{zerolocuspep}] Since the intersection of finitely many cosets in $\ZZ^r$ is still a coset (or the empty set), and if all these cosets have rank $<r$, so does the intersection, we may assume $\mathbf{f}=(f)$, i.e.\  $s=1$. Write $f(\mathbf{x})=\sum\limits_{i=1}^h a_i \boldsymbol{\lambda}^{A_i \cdot \mathbf{x}^T} $ where $a_i\neq 0$ and $A_i \in M_{r\times k}(\ZZ)$. Take $K,S$ such that all coefficients and bases are $S$-units. Then the set of zeros of $f$, i.e.\  those $\mathbf{n}\in \ZZ^r$ such that
		$$\sum\limits_{i=1}^h a_i \boldsymbol{\lambda}^{A_i \cdot \mathbf{n}^T}=0$$
		is equal to the preimage under the group homomorphism 
		$$\varphi\colon \ZZ^r\to \cO_S^h,\text{  }\mathbf{n}\mapsto \left( \boldsymbol{\lambda}^{A_1 \cdot \mathbf{n}^T},\dots, \boldsymbol{\lambda}^{A_h \cdot \mathbf{n}^T}  \right)$$
		of the set of solutions of the homogeneous $S$-unit equation
		$$\sum\limits_{i=1}^h a_i s_i=0.$$
		We know by Theorem \ref{sunitfinitecoset} that the latter solution set is equal to a finite union of cosets in $(\cO_S^*)^h$. We can therefore obtain the first assertion because of the fact that preimages of cosets under group homomorphisms are still cosets. 
		
		Now suppose $\lambda_1,\dots,\lambda_k$ are multiplicatively independent, let $\{1,\dots,h\}=I_1\sqcup \cdots \sqcup I_l$ be a partition such that $|I_j|\geq 2$ for all $j$; then the preimage
		$$\varphi^{-1} \left(\prod_{j=1}^l (\cO_S^*)_{\Delta_j}\right)$$
		is a subgroup of $\ZZ^r$ with rank $<r$. This yields the second assertion. 
	\end{proof}
	Next we study the cardinality of fibers of (PEP)s. The following theorem essentially characterizes intersections of (PEP) sets.
		\begin{theo}\label{fiberpep} Let $\mathbf{f}, \mathbf{g}:\mathbb{Z}^r \to \overline{K}^s$ be vectors of purely exponential polynomials, with $\mathbf{g}$ reduced. Then, outside of a finite union of cosets of rank $\leq r-1$, the function
		\[
		  \ZZ^r \to \ZZ_{\geq 0} \cup \{\infty\}, \ \ \u{n} \mapsto \# \mathbf{g}^{-1}(\mathbf{f}(\u{n}))
		\]
		attains finite values and depends only on $\u{n} \bmod N$, for some natural number $N=N(\mathbf{f}, \mathbf{g}) \geq 1$.
	\end{theo}

	In particular, taking $\mathbf{f}=\mathbf{g}$, we get the following useful corollary.
	
	\begin{coro}[Generic Fiber Theorem]\label{genericfiberthm}
		Let $\mathbf{f}:\mathbb{Z}^r \to \overline{K}^s$ be a  reduced vector of purely exponential polynomials. Then, there exist (i) a finite partition $\ZZ^r=C_1\sqcup \cdots \sqcup C_l$ into finite index cosets; (ii) for each $i$, a subset $\Theta_i \subset C_i$ which is equal to a finite union of  cosets of rank $\leq r-1$, such that for each $i$, there is $\nu_i \in \NN$ such that
$$\#\{\mathbf{n}' \in C_i \backslash \Theta_i\colon \mathbf{f}(\mathbf{n})=\mathbf{f}(\mathbf{n'})\}\equiv \nu_i\text{  for all  }\mathbf{n}\in C_i \backslash \Theta_i.$$

	\end{coro}

	\begin{proof}[Proof of Theorem \ref{fiberpep}]   By the Zero Locus Theorem \ref{zerolocuspep}, the set 
		$$W\coloneqq \left\{ (\mathbf{n},\mathbf{m}) \in \mathbb{Z}^r \times \mathbb{Z}^r \colon \mathbf{g}(\mathbf{n})=\mathbf{f}(\mathbf{m})\right\}$$
		is equal to a finite union of cosets $L_1 \cup \cdots \cup L_k \subseteq \mathbb{Z}^r \times \mathbb{Z}^r$. 
		
		Let $\mathrm{p}_2\colon \mathbb{Z}^r \times \mathbb{Z}^r\to \mathbb{Z}^r$ denote the projection to the second factor. Let
		$$M_1:=\{i\colon \ \mathrm{p}_2(L_i)\text{ has rank }r\}\text{ and }M_2:=\{i\colon\ \mathrm{p}_2(L_i)\text{ has rank }<r\}.$$

		Let $i\in M_1$. We claim that $\rk L_i=r$. After passing to a finite index coset of $\ZZ^r$, we may assume without loss of generality that the second projection is surjective. Note that, by standard linear algebra, if we assume by contradiction that $\rk L_i>r$, then there would exist an affine section of $\mathrm{p}_2|_{L_i}\colon L_i \to \ZZ^r$ such that the projection of this section onto the first factor has rank $<r$. In particular, if we denote this section by $\sigma=(\sigma_1,id)\colon \ZZ^r \to L_i$, we would have that $\ker \sigma_1 \neq \{0\}$. Since, for every $\mathbf{m} \in \ZZ^r$ we have $\mathbf{f}(\mathbf{m})=\mathbf{g}(\sigma_1(\mathbf{m}))$, we would deduce that $\ker \sigma_1$ is annihilated by all the {\em exponents} of $\mathbf{g}$ (recall that these are the linear forms appearing in the exponents of the monomials of the components of $\mathbf{g}$ as in \eqref{Eq:expressionPEP}), but these linear forms span the dual space of $\QQ^r$ over $\QQ$ by assumption, and thus have no non-trivial common zero. This leads to a contradiction and proves the claim.
		
		Thus, for $\mathbf{n} \in \mathbb{Z}^r$ lying outside $\bigcup_{i \in M_2}( \mathrm{p}_2(L_i) \cup \bigcup_{j \neq i}\mathrm{p}_2(L_i \cap L_j))$, which is a finite union of cosets of rank $\leq r-1$, the cardinality of the fiber $\mathbf{g}^{-1}(\mathbf{f}(\mathbf{n}))$ is equal to the number of indexes $i\in M_1$ such that $\mathbf{n} \in \mathrm{p}_2(L_i)$.
		
		Taking $N$ to be any natural number that divides the (finitely many) indexes of the cosets $\{\mathrm{p}_2(L_i) \subseteq \mathbb{Z}^r\}_{i \in M_1} $, we obtain the sought claim.
	\end{proof}
	While it seems that a general (PEP) set does not have any simple structure, if we know that a (PEP) set in $K^*$ is a semi-group, then it can be easily described:
	\begin{theo}[Image Theorem] \label{imagetheorem} Let $K$ be a number field and $\Sigma \subset K^*$ be a (PEP) set, suppose that $\Sigma$ is a multiplicative semi-group; then this semi-group is actually a group which is finitely generated.
	\end{theo}
	
	\begin{proof} Write $\Sigma =f(\ZZ^r)$ where $f$ is a purely exponential polynomial with shape $f(\mathbf{x})=u_1(\mathbf{x})+\cdots + u_e (\mathbf{x})$ where 
		$$u_i(\mathbf{x})=a_i \cdot \lambda_1^{l_{i,1}(\mathbf{x})}\cdots \lambda_k^{l_{i,k}(\mathbf{x})}.$$
		Let $S\subset V^K$ be a finite set of places containing all infinite ones such that all bases and coefficients of $f$ are $S$-units.  Take any $b\in \Sigma$, then $b^h\in \Sigma,\forall h\in \NN$ because $\Sigma$ is a semi-group. Let $S'\supseteq S$ be a finite set of places so that $b$ is an $S'$-unit. We obtain that for each $h\in \NN$, there exists some $\mathbf{n}_h\in \ZZ^r$ such that $\left(  \frac{u_1(\mathbf{n}_h)}{b^h},\dots, \frac{u_e(\mathbf{n}_h)}{b^h} \right)$ is a solution of the $S'$-unit equation $x_1+\cdots+x_h=1$. Recall that for this $S'$-unit equation, there is a finite set $\Phi=\Phi(b)\subset \cO_{S'}^*$ such that for each solution, there is at least one coordinate belonging to $\Phi$, cf. \cite[Remark 2.5 (iv)]{CorvajaZannier}. By the pigeonhole principle, there is an index $i$ such that there are infinitely many $h \in  \NN$ such that $\frac{u_i(\mathbf{n}_h)}{b^h}$ is equal to the same element $\theta\in \Phi$ (in fact, we only need two different such $h \in \NN$). Take any $v\not\in S$, then $1=|u_i(\mathbf{n}_h)|_v=|b|_v^h\cdot |\theta|_v$ holds for infinitely many $h$, thus $|b|_v=1$, which implies that $b\in \cO_S^*$. So $S'=S$ and $\Sigma \subset \cO_S^*$.
		
		Let $N$ be an exponent for the torsion subgroup of $\langle \lambda_1,\dots,\lambda_k \rangle$. Since the group generated by $\lambda_1^N,\dots,\lambda_k^N$ is torsion free, we can choose a free basis $\mu_1,\dots,\mu_t$ of it. Thus, for any  $\mathbf{c}\in \{0,1,\dots, N-1\}^r$, the new purely exponential polynomial in $r$ variables defined by 
		$$f_{\mathbf{c}}\colon \mathbf{x}=(x_1,\dots,x_r)\mapsto f(N\mathbf{x}+\mathbf{c})$$
		can be rewritten so that it has multiplicatively independent bases. Let us write $f_{\mathbf{c}}=\sum\limits_{i=1}^{e_{\mathbf{c}}} u_{i,\mathbf{c}}(x_1,\dots,x_r)$, where the $u_{i,\mathbf{c}}$'s are distinct purely exponential monomials in $\mu_1,\dots,\mu_t$.
		
		Now for any $\mathbf{n}\in \ZZ^r$, $(-f_{\mathbf{c}}(\mathbf{n}),u_{1,\mathbf{c}}(\mathbf{n}),\dots,u_{e_{\mathbf{c}},\mathbf{c}}(\mathbf{n}))$ is a solution of the homogeneous $S$-unit equation $x_0+x_1+\cdots +x_{e_{\mathbf{c}}}=0$.
		Again, according to \cite[Remark 2.5 (iv)]{CorvajaZannier}, there is a fixed finite set $\Phi \subset \cO_S^*$ such that for each $\mathbf{n}\in \ZZ^r$, two distinct terms among $-f_{\mathbf{c}}(\mathbf{n}),u_{1,\mathbf{c}}(\mathbf{n}),\dots,u_{e_{\mathbf{c}},\mathbf{c}}(\mathbf{n})$ have ratio in $\Phi$. Thus for any $\mathbf{n}\in \ZZ^r$, either $\frac{-f_{\mathbf{c}}(\mathbf{n})}{u_{i,\mathbf{c}}(\mathbf{n})}=\theta \in \Phi$ for some $i$ or $\frac{u_{i,\mathbf{c}}(\mathbf{n})}{u_{j,\mathbf{c}}(\mathbf{n})}=\theta \in \Phi$ for some $i\neq j$. Since the bases are multiplicatively independent, the set
		$$\{\mathbf{n}\in \ZZ^r\colon u_{i,\mathbf{c}}(\mathbf{n})=\theta u_{j,\mathbf{c}}(\mathbf{n}) \}$$
		is a coset of $\ZZ^r$ of rank $<r$. Moreover, by the Zero Locus Theorem \ref{zerolocuspep}, the set
		$$\{ \mathbf{n}\in \ZZ^r\colon -f_{\mathbf{c}}(\mathbf{n})=\theta u_{i,\mathbf{c}}(\mathbf{n}) \}$$
		is equal to a finite union of cosets of $\ZZ^r$ of rank $<r$, if $e_{\mathbf{c}}>1$. Since $\ZZ^r$ cannot be covered by finitely many cosets of $\ZZ^r$ of rank $<r$, we conclude that $e_{\mathbf{c}}=1$, i.e.\  $f_{\mathbf{c}}$ is a purely exponential monomial. Thus $f_{\mathbf{c}}(\ZZ^r)$ is a translate of a finitely generated multipicative subgroup of $\langle \mu_1,\dots,\mu_t\rangle$. Since this holds for all $\mathbf{c}$, and $\Sigma= \bigcup_{\mathbf{c}}f_{\mathbf{c}}(\ZZ^r)$, and then the following lemma implies that $\Sigma$ is a group, hence a finitely generated multiplicative group.
	\end{proof}
	
	\begin{lemma}
		Let $C_1,\ldots,C_r$ be cosets in a commutative group $G$. If $C_1 \cup \cdots \cup C_r$ is a semi-group, then it is a group.
	\end{lemma}
	\begin{proof}
		It is enough to prove that for every $\lambda\in\Sigma,  \lambda^\ZZ \subset \Sigma$. Let $\lambda \in \Sigma \coloneqq C_1 \cup \cdots \cup C_r$. We may assume that $\lambda$ is non-torsion. Then we have that:
		\[
		\lambda^\NN \subseteq (C_1 \cap \lambda^\ZZ) \cup \cdots \cup (C_r \cap \lambda^\ZZ) \subseteq \lambda^\ZZ,
		\]
		so there is an $i$ such that the coset $C_i \cap \lambda^\ZZ \subseteq \lambda^\ZZ$ is infinite. Thus the intersection $\Sigma \cap \lambda^\ZZ$ is a semi-group of the cyclic group $\lambda^\ZZ$ that contains both the generator $\lambda$ and a finite index coset, and hence is in fact equal to the whole $\lambda^\ZZ$. This concludes the proof.
	\end{proof}
	
	\begin{rema}\label{gpimagenopep} {\rm It is possible that a purely exponential polynomial which is not a purely exponential monomial still has a group image. For example, let $\omega$ be a cubic root of unity, and let $f(m,n,k)=\omega^m 2^n+\omega^k 2^n$. Then $f$ has image equal to the multiplicative group generated by $2,-1,\omega$.  Anyway, it is important to notice that according to the Image Theorem, if a (PEP) set in $K$ is a semi-group, then it is equal to $g(\ZZ^r)$, where $g$ is a certain purely exponential \underline{monomial}.}
	\end{rema}

        Finally, we derive a crucial decomposition result which claims that every (PEP) set is a finite union of reduced ones.

    \begin{prop}\label{decomposepep}
        Every (PEP) set $\Sigma \subseteq K^s=\mathbf{f}(\ZZ^r)$ of $r$ variables is equal to a union $\Sigma_1\cup \cdots \cup \Sigma_l$ of finitely many (PEP) sets $\Sigma_i$, each of which is the image of a \underline{reduced} vector of purely exponential polynomials of \underline{at most}  $ r$ variables.
    \end{prop}

    \begin{proof}
    
    Let us write
	$$f_i(\mathbf{n})=\sum _{A\in J_i}a_{A,i} \u{\lambda} ^{A\cdot \mathbf{n}^T}=\sum _{A\in J_i} a_{A,i} \lambda_1^{\mathbf{r}_1\cdot \mathbf{n}^T}\cdots \lambda_k^{\mathbf{r}_k\cdot \mathbf{n}^T},$$
	where $J_i$ is a finite subset of $M_{k \times r}(\ZZ)$, $a_{A,i}\neq 0$, and $\mathbf{r}_1,\dots , \mathbf{r}_k$ are rows of $A$. Let $J=J_1\cup \cdots \cup J_s$. We divide the proof in two ``reductions''.
	\vskip5mm
	\noindent
	{\bf First Reduction.} We reduce here to the case where the bases of $\mathbf{f}$ are multiplicatively independent. The trick in the following is similar to the one in the proof of the Image Theorem \ref{imagetheorem}.
	
	 Let $\Lambda$ be the subgroup of $K^*$ generated by $\lambda_1,\dots, \lambda_k$, and let $E$ be the exponent of the torsion part of $\Lambda$. We may divide $\ZZ^r$ into a finite disjoint union of cosets
	$$\ZZ^r=\bigsqcup_{\mathbf{a}\in \{0,1,\dots,E-1\}^r}M_{\mathbf{a}}$$
	where $M_{\mathbf{a}}=E\ZZ^r+\mathbf{a}$. Now for a fixed $\mathbf{a}\in \{0,1,\dots,E-1\}^r$ consider the purely exponential polynomial 
	$$\mathbf{f}_{\mathbf{a}}\colon \ZZ^r\to K^s,\text{  }\mathbf{n} \mapsto \mathbf{f}(E\mathbf{n}+\mathbf{a}).$$
	Notice that $\mathbf{f}_{\mathbf{a}}=(f_{1,\mathbf{a}},\dots, f_{s,\mathbf{a}})$ may be rewritten as 
	$$f_{i,\mathbf{a}}=\sum _{A\in J_i}a_{A,i}\lambda_1^{\mathbf{r}_1\cdot \mathbf{a}^T}\cdots \lambda_k^{\mathbf{r}_k\cdot \mathbf{a}^T}\cdot (\lambda_1^E)^{\mathbf{r}_1\cdot \mathbf{n}^T}\cdots (\lambda_k^E)^{\mathbf{r}_k\cdot \mathbf{n}^T},$$
	which is a vector of purely exponential polynomials whose bases are $\lambda_i^E, i=1,\ldots,k$. Note now that $\Lambda_{\mathbf{a}}\coloneqq \langle \lambda_1^E,\dots, \lambda_k^E\rangle$ is torsion free, choosing a $\ZZ$-basis $\mu_1,\dots,\mu_l$ of it, we may rewrite $\mathbf{f}_{\mathbf{a}}$ as a vector of purely exponential polynomials with {\bf multiplicatively independent} bases $\mu_1,\dots,\mu_l$.
	
	Noting that $\Sigma=\bigcup_{\mathbf{a}\in \{0,1,\dots,E-1\}^r}\mathbf{f}_{\mathbf{a}}(\ZZ^r)$, it suffices to establish the theorem for $\mathbf{f}_{\mathbf{a}}$, for each of the finitely many $\mathbf{a}$'s.
	\vskip5mm
	\noindent
	{\bf Second Reduction.} Assuming the first reduction is done, we further reduce here to the case where $\dim_{\QQ}\sum \mathbb{Q}\cdot \mathbf{r}_i=r$ (where the sum is taken for all $\mathbf{r}_i$ as a row vector of some $A\in J$), i.e.\  $\mathbf{f}$ is reduced. Let $H$ be the subgroup of $\ZZ^r$ defined as
	$$H=\bigcap\limits_{A\in J}\{\mathbf{n}\in \ZZ^r\colon A \cdot \mathbf{n}^T=\mathbf{0}^T \}.$$
	
	Since $H$ is saturated in $\ZZ^r$ (meaning that if $\mathbf{w} \in\ZZ^r $ is a vector such that there exists an integer $k>0$ with $k \mathbf{w} \in H$, then  $\mathbf{w}\in H$), after changing the basis of $\ZZ^r$ (this may always be done without loss of generality), we may assume that $H=\ZZ^q\times \{0\}^{r-q}, \ q\leq r$. Equivalently, $\langle\langle \mathbf{r}_1,\dots, \mathbf{r}_k\rangle, {A\in J}\rangle_{\QQ}=\{0\}^q \times \QQ^{r-q}.$
	
	Now define 
	$$\mathbf{f}_{\mathrm{new}}\colon \ZZ^{r-q}\to K^s,\text{   }(x_{q+1},\dots, x_r)\mapsto \mathbf{f}\left(\{0\}^q\times (x_{q+1},\dots, x_r)\right),$$
	which is reduced. Note that $\mathbf{f}$ factors through $\mathbf{f}_{\mathrm{new}}$ as $\mathbf{f}:\ZZ^r \to \ZZ^{r-q} \xrightarrow{\mathbf{f}_{\mathrm{new}}} K^s$, and so $\mathbf{f}(\ZZ^r)=\mathbf{f}_{\mathrm{new}}(\ZZ^{r-q})$.

  	\vskip3mm

    The two reductions above combine to give the lemma.

    \end{proof}

	\section{The asymptotic behavior of general (PEP) and its proof}\label{pfht}
	In this section we state and prove a more precise version (Theorem \ref{mainlongpaper}) of Theorem \ref{minspecial}, which serves as the main asymptotic result about (PEP) in this paper.
	
	Let $\mathbf{f}=(f_1,\dots,f_s)$ be a vector of purely exponential polynomials in $r$ variables. 
	
	\begin{defi} {\rm For a coset $M\subset \ZZ^r$, we define the {\bf stabilizer} of the restriction $\mathbf{f}| _M$, denoted by $\stab _M \mathbf{f}$, to be the maximal subgroup $\mathcal{K}$ of $M-M$ such that $\mathbf{f}| _M$ is invariant under translations by $\mathcal{K}$, i.e.
		$$\mathbf{f}(\mathbf{m}+\mathbf{k})=\mathbf{f}(\mathbf{m})\text{ for all }\mathbf{m}\in M\text{ and }\mathbf{k}\in \mathcal{K}.$$
  Moreover, for $\mathbf{f}=(f_1,\dots,f_s)$, we have the obvious relation
		$$\stab _M \mathbf{f}=\bigcap _{i=1}^s \stab _M f_i.$$}
	\end{defi}

   By Artin's linear independence of characters (remember our convention that the characters appearing in each $f_i$ are pairwise different), it is plain that the stabilizer $\stab_{\ZZ^r} \mathbf{f}$ of a (PEP) $\mathbf{f}$ coincides with the $\mathcal{K}$ of Proposition \ref{Prop:norm2}.
	
	\begin{rema}{\rm
		Notice that $\stab_M \mathbf{f}=\{\mathbf{0}\}$ if $\mathbf{f}|_M$ is reduced, where, to use the adjective reduced, we are viewing the latter as a vector of purely exponential polynomials by choosing a base-point $m_0 \in M$ and by representing elements of $M$ using coordinates for the $\ZZ$-module $M-m_0$. More generally, if the bases of $\mathbf{f}|_M$ are multiplicatively independent, then $\stab_M \mathbf{f}$ is equal to the intersection of the kernels of all exponents of $\mathbf{f}|_M$. However, it is worth observing that there do exist non-reduced (PEP) with trivial stabilizer, e.g.\  $(-1)^n+2^n$.}
	\end{rema}
	
	Our main asymptotic result on (PEP) reads:
	\begin{theo}[Main asymptotic result: precise version]\label{mainlongpaper} There exists a finite partition $\ZZ^r =L_1\sqcup \cdots \sqcup L_N$, where each $L_i$ is a quasi-coset $M_i\backslash N_i$, where $M_i$ is a coset of $\ZZ^r$ and $N_i$ is a finite union of cosets inside $M_i$ of smaller rank, such that for each $i=1,\dots,N$, writing $\mathcal{K}_i=\stab_{M_i}\mathbf{f}$, there exists a norm $\|\cdot\|_i$ on $((M_i-M_i)/\mathcal{K}_i)\otimes \R$ such that
		\begin{equation}\label{Mainineq}
		h_{\mathrm{aff}}(\mathbf{f}(\mathbf{n}))\sim \|\mathbf{n}-\mathbf{n}_0\|_{i} \text{ for }\mathbf{n}\in L_i,
		\end{equation}
    where $\mathbf{n}_0$ is any element of $M_i$. 
	\end{theo}
Here both the norms $\|\cdot \|_i$ and the stabilizers $\mathcal{K}_i$ lie in effective finite sets, but the error term implicit in $\sim$ above is unattainable with our method. (See the discussion right after the proof of Theorem \ref{mainlongpaper}.) The effectivity  enables us to obtain the following corollary.
    
	\begin{coro}\label{CoromainThm}
		There exists a partition as above and an effective $C>0$ such that for each $i=1,\dots,N$, the solutions of
		\begin{equation}\label{MainCorineq}h_{\mathrm{aff}}(\mathbf{f}(\mathbf{n}))\leq C\cdot \min _{\mathbf{v}\in \mathcal{K}_i}\|\mathbf{n}+\mathbf{v}\|_{\infty} \text{ for }\mathbf{n}\in L_i
        \end{equation}
		lie in a union of finitely many translates of $\mathcal{K}_i$. 
	\end{coro}
    
    Here the number of translates is ineffective. We prove the corollary at the end of this section.

    When $\mathbf{f}$ is a reduced (PEP), Proposition \ref{Prop:norm} gives already the following clean statement as a corollary. This may also be seen as a ``base case'' of Corollary \ref{CoromainThm} (see Remark \ref{Rmk:mainlongpaper}).
     
	\begin{coro}\label{mainhtsteponecoro}
	    Suppose $\mathbf{f}$ is a reduced (PEP); then there exists an effective $C'>0$ such that the inequality
	\begin{equation}\label{mainhtstepone}
	    h_{\mathrm{aff}}(\mathbf{f}(\mathbf{n})) \leq C'\cdot  \| \u{n} \|_{\infty}
	\end{equation}
	has only finitely many $\mathbf{f}$-non-degenerate solutions. 
	\end{coro}
 Recall that ``non-degenerate'' was defined in Definition \ref{defdegenerate}. Needless to say, the number of non-degenerate solutions here is ineffective.

    \begin{rema}\label{Rmk:mainlongpaper} {\rm 
    The partition $\ZZ^r =L_1\sqcup \cdots \sqcup L_N$ in the theorem and the corollary above comes by classifying the elements $\mathbf{n}\in \ZZ^r$ with respect to the ``degeneracy type'' of the sums $f_i(\u{n}), i=1,\ldots,r$ in the sense as in the end of the introduction.

    The proof of Theorem \ref{mainlongpaper} yields, in particular, that when $\mathbf{f}$ is reduced, one of the cosets $M_1,\dots, M_N$ is equal to $\ZZ^r$ whose corresponding quasi-coset contains all $\mathbf{f}$-non-degenerate $\mathbf{n}$, while the others have rank $<r-1$. 
    }
    \end{rema}

\begin{proof}[Proof of Theorem \ref{mainlongpaper}]
    For each $\u{n} \in \ZZ^r$, and each $i=1,\ldots,s$, we may express the sum $f_i$ as a sum of two (possibly empty) subsums:
    \begin{equation}\label{Decompositionfi}
      f_i=f_{i,0}+f_{i,1},
    \end{equation}
    where $f_{i,1}(\u{n})=0$, and $f_{i,0}(\u{n})$ is non-degenerate (or empty). (i.e.\ we may choose a ``degeneracy type'' of $\mathbf{f}(\mathbf{n})$.)

    Motivated by this observation, we define, for each of the finitely many decompositions $D$ of the sums $f_i$ as in \eqref{Decompositionfi}:
    \[
    \Theta_D:=\left\{ \u{n} \in \ZZ^r \mid f_{i,1}(\u{n})=0 \text{ and }f_{i,0}(\u{n}) \text{ is non-degenerate, for each }i\text{ with }\mathbf{f}_i(\mathbf{n})\neq 0\right\}.
    \]

    According to the observation above, we have that $\ZZ^r=\bigcup_D \Theta_D$.
    Note that, for each $D$, the set $\Theta_D$ is of the form ``finite union of cosets minus finite union of cosets'' by the Zero Locus Theorem \ref{zerolocuspep} (when the bases of $\mathbf{f}$ are multiplicatively independent, all the cosets appearing here have rank $<r$ unless $D$ is the trivial decomposition, i.e.\ $f_{i,1}=0$ for all $i$). Moreover, any set of this form may be written as a finite union of {\it quasi-cosets}, where we define a quasi-coset to be a coset minus a finite union of subcosets of {\bf smaller} rank.
   
    Fix now $D$ and let $\chi_1:=\boldsymbol{\lambda}^{A_1\cdot \u{n}^T},\ldots,\chi_t:=\boldsymbol{\lambda}^{A_t\cdot \u{n}^T}$ be the different characters appearing in the sums $f_{1,0}, \ldots, f_{s,0}$.
    Let also $M \setminus N$ be a maximum quasi-coset in $\Theta_D$, and fix an $\u{n}_0 \in M \setminus N$. Consider now the restriction ${\mathbf{f}}|_M=(f_{1,0},\ldots,f_{s,0})|_M$, which we may identify with a purely exponential polynomial after fixing an affine isomorphism $M \cong M-\u{n}_0 =M-M \cong \ZZ^{\rk M}$. 

    Writing, for each $i$, $f_{i,0}= \sum\limits_{j=1}^{t(i)} a_{i,j}\chi_{i,j}$ with $\{\chi_{i,j}: j=1,\ldots, t(i)\}\subseteq \{\chi_1,\ldots,\chi_t\}$ and $a_{i,j}\neq 0$, we have that
    \[
    f_{i,0}|_M(\mathbf{n})= \sum_{j=1}^{t(i)} \left(a_{i,j}\chi_{i,j}(\mathbf{n}_0)\right)\chi_{i,j}|_{M-M}(\mathbf{n}) \text{ for }\mathbf{n}\in \ZZ^{\mathrm{rk}M}. 
    \]
    The characters $\chi_{i,j}|_{M-M},j=1,\ldots, t(i)$ are not necessarily pairwise different, but we may cluster the identical ones together and write
    \[
    f_{i,0}|_M= \sum_{\chi' \in \mathcal{S}_i}\left(\sum_{\substack{j=1 \\ \chi_{i,j}=\chi'}}^{t(i)} a_{i,j}\chi_{i,j}(\mathbf{n}_0)\right)\chi',
    \]
    where $\mathcal{S}_i:=\{\chi_{i,j}|_{M-M},j=1,\ldots, t(i)\}$.
    It is important to notice that no character $\chi'$ gets simplified in the following sense: \ $\sum\limits_{j,\chi_{i,j}=\chi'} a_{i,j}\chi_{i,j}(\mathbf{n}_0) \neq 0$ for every $\chi' \in \mathcal{S}_i$, for otherwise the sum $f_{i,0}|_M$ would be degenerate on the whole of $M$, and in particular on $M \setminus N$ as well. Summarizing, even though some of the characters $\chi_1|_{M-M}, \ldots, \chi_t|_{M-M}$ may coincide, each of them appears at least once in the purely exponential polynomials $f_{1,0}|_M, \ldots,f_{s,0}|_M$.

    Hence Propositions \ref{Prop:norm} and \ref{Prop:norm2} imply that
    \[
    h_{\mathrm{aff}}({\mathbf{f}}(\u{n}))=h_{\mathrm{aff}}({\mathbf{f}}|_M(\u{n}))\sim h_{\mathrm{aff}}((\chi'(\mathbf{n}))_{\chi'\in \mathcal{S}_1 \cup \cdots \cup \mathcal{S}_s})=h_{\mathrm{aff}}(\chi_1|_{M-M}(\mathbf{n}), \ldots, \chi_t|_{M-M}(\mathbf{n}))
    \]
    for $\mathbf{n}\in M \setminus N$,
    and that the right hand side defines a norm on $((M-M)/\mathcal{K})\otimes \RR$, where $\mathcal{K}$ is the stabilizer of $\mathbf{f}|_M$.    

    For each $D$, writing $\Theta_D$ as the union of its (finitely many) maximal quasi-cosets, and adjoining these unions as $D$ varies through all (the finitely many) decompositions of the sums $f_i$, we obtain a decomposition $\ZZ^r=L_1\cup \cdots \cup L_N$ as in the statement (except that the $L_i$ might intersect each other here). Refining the decomposition $\ZZ^r=L_1\cup \cdots \cup L_N$ to a partition concludes the proof.
\end{proof}

{\bf Effectivity.}
As seen from the proof above, the norms $\|\cdot\|_i$ and the stabilizers $\mathcal{K}_i$ are effective in the following sense. For each decomposition $D$, the groups $M-M$ underlying the maximal quasi-cosets are effectively computable (i.e.\ each of them belongs to a finite family of effectively computable subgroups of $\ZZ^r$) by the Zero Locus Theorem \ref{zerolocuspep}. Then it is clear that, for each of these groups, the stabilizer of $\mathcal{K}$ of $\mathbf{f}|_M$ is equal to $\bigcap_{j=1}^t \ker \chi_j|_{M-M}$, and that
the semi-norm $\|\mathbf{n}\|\coloneqq h_{\mathrm{aff}}(\chi_1|_{M-M}(\mathbf{n}), \ldots, \chi_t|_{M-M}(\mathbf{n})), \ \mathbf{n} \in M-M$ is effective, as it is given explicitly.

\begin{proof}[Proof of Corollary \ref{CoromainThm}]
    Consider, for each decomposition $D$, and each subgroup $H \subseteq \ZZ^r$ that may potentially appear as a difference $M-M$, the semi-norm:
    \[
    \|\mathbf{n}\|_{H}\coloneqq h_{\mathrm{aff}}(\chi_1|_{H}(\mathbf{n}), \ldots, \chi_t|_{G}(\mathbf{n})), \ \ \mathbf{n} \in H
    \]
    with kernel $\mathcal{K}_H\coloneqq \bigcap_{j=1}^t \ker \chi_j|_{H}$. Let $C_D(H)>0$ be the comparison constant such that $\|\mathbf{n}\|_{H} \geq C_D(H) \cdot \min_{\mathbf{v} \in \mathcal{K}_H} \|\mathbf{n}\|_{\infty}. $
    Choose $C\coloneqq \frac12 \min_{D,H} C_D(H)$. The corollary is now clear.
\end{proof}

	\section{Abundance of points in \texorpdfstring{$S$}{}-arithmetic group} \label{vol}

	This section consists mainly group-theoretic arguments. We shall verify Theorem \ref{volestimate} which provides an abundance property for points in many ambient $S$-arithmetic groups, establishing as a byproduct the sparseness of their (PEP) subsets by Theorem \ref{firstmainthm}. We also sketch an alternative, more combinatorial, approach in Remark \ref{Rmk: a second approach}.
  
    Let $K$ be a number field, and let $S$ be a finite set of places containing all archimedean ones. Let $G$ be a linear algebraic group over $K$, and let $\rho$ be a $K$-representation $\rho\colon G \rightarrow \GL_{n,K} $ with \underline{finite kernel}, we first deal with the case when $G$ and $\Gamma$ satisfy the following.
 \begin{itemize}
     \item[($\mathcal{C}_1$)] The ambient algebraic group $G$ is $K$-simple and simply connected;
     \item[($\mathcal{C}_2$)] The $S$-arithmetic group $\Gamma$ is equal to $G(\cO_S)\coloneqq G(K)\cap  \rho^{-1}(\GL_n(\cO_S))$.
 \end{itemize}
 We do assume both ($\mathcal{C}_1$) and ($\mathcal{C}_2$) until the end of the section where we will reduce the general case to it. Because of ($\mathcal{C}_1$), $G$ is $K$-simple and thus the induced map $\det \circ \rho\colon G\to \mathbb{G}_{m,K}$ is constant, hence the image of $\rho$ lies in $\mathrm{SL}_{n,K}$. Here we emphasize that we take $\rho$ to be a representation with finite kernel instead of a faithful one, this flexibility allows us to deal with the later reduction steps more easily.

	By assumption ($\mathcal{C}_1$), there is no non-trivial $K$-defined character of $G$, thus $\Gamma =G(\cO_S)$ is a lattice in $G_S\coloneqq\prod_{v \in S} G(K_v)$ cf. \cite[Theorem 5.7]{PR}. We always \emph{assume} that $\Gamma$ is infinite, which is the only case making the sparseness question (i.e.\ Theorem \ref{volestimate}) non-trivial. Recall that $\Gamma$ being infinite is equivalent to the existence of a $v \in S$ for which $G(K_v)$ is non-compact, or, equivalently, $G$ is $K_v$-isotropic \cite[Theorem 3.1]{PR}. 

       Recall that there is a variety of works about counting lattice points in $S$-arithmetic groups (or discrete subgroups of Lie groups) claiming that the number of lattice points in a certain height ball (or norm ball) is proportional to the volume of the height ball (or norm ball). See the pioneering work of Duke-Rudnik-Sarnak \cite{DukeRudnickSarnak} and Eskin-McMullen \cite{EM93}, as well as follow-up papers e.g.\  \cite{GMO,GW07,GN12,Maucourant07}. Among them, we will use a counting result of Gorodnik-Nevo \cite[Theorem 5.1]{GN12}, whose form is most convenient for our purpose. We first introduce their set-up, especially a variation $\widetilde{H}$ of our height function $H_{\mathrm{aff}}$.  Their height $\widetilde{H}\colon G_S\to \RR_{+}$ is defined as the product $\prod\limits_{v \in S} \widetilde{H}_v(\rho(g_v))^{1\slash [K\colon \mathbb{Q}]}$, where $\widetilde{H}_v$ denotes the local height function on $K_v^{n^2}=M_n(K)$ defined as follows

\begin{equation}\label{deflocalht}
\widetilde{H}_{v}(\mathbf{x})=\widetilde{H}_{v}(x_{11},x_{12} \ldots, x_{nn})= \left\{
  \begin{array}{ll}
    \left(|x_{11}|_{v}^{2}+\cdots+|x_{nn}|_{v}^{2}\right)^{[K_v\colon \R] / 2}, & \hbox{if $v\in V^K_{\infty}$,} \\
    \max \left\{\left\|x_{11}\right\|_{v}, \ldots,\left\|x_{nn}\right\|_{v}\right\}, & \hbox{if $v\in S\cap V^K_f$,} 
  \end{array}
\right.
\end{equation}
where $|\cdot|_{v}$ denotes the standard archimedean absolute value (before normalization, i.e.\ $\norm{\cdot}_v=| \cdot |_v^{[K_v \colon \RR]}$).

	Let $\mu$ be the Haar measure on $G_S$, normalized so that $\mu(G_S/\Gamma)=1$; then the Gorodnik-Nevo counting theorem  \cite[Theorem 5.1]{GN12} implies

       \begin{theo}[Gorodnik-Nevo]\label{GNcounting} 
           \begin{equation}\label{Estimate}
	\#\{\gamma \in \Gamma \mid \widetilde{H}(\rho(\gamma)) \leq T\}\sim \mu(B_T), \text{ as } T \to \infty,
	\end{equation}
	where $$B_{T}\coloneqq\{g\in G_S\colon  \widetilde{H}(g) \leq T\}.$$
       \end{theo}
	
	Recall that in our setting, for $\mathbf{x}=\left(x_{11},x_{12}, \ldots, x_{nn}\right) \in K^{n^2}$, the affine height function is defined as
	\[
	H_{\mathrm{aff}}(\mathbf{x}):= \left(\prod_{v \in V^K} \max \left\{\left\|x_{11}\right\|_{v}, \ldots,\left\|x_{nn}\right\|_{v},1\right\}\right)^{1\slash [K\colon \mathbb{Q}]}.
	\]
	To apply Theorem \ref{GNcounting} in our setting, we need the first inequality of the following comparison result (the second is included only for completeness).
	\begin{lemma}\label{HtildacompatiblewithH}
		For $\gamma \in \Gamma$, we have inequalities
		\[
		H_{\mathrm{aff}}(\rho(\gamma)) \leq \widetilde{H}(\rho(\gamma)) \leq {n} \cdot H_{\mathrm{aff}}(\rho(\gamma)).
		\]
	\end{lemma}
	\begin{proof}
		For any $\gamma \in \Gamma$, write $\mathbf{x}=(x_{ij})_{1\leq i,j \leq n}\coloneqq \rho(\gamma)\in \mathrm{SL}_n(\cO_S)$. Since $\mathbf{x}\in \mathrm{SL}_n(\cO_S) \subset \cO_S^{n^2}$, we have
  $$\left\|x_{i}\right\|_{v}\leq 1,\forall i\text{ and }v\notin S\Longrightarrow\max \left\{\left\|x_{11}\right\|_{v} \ldots,\left\|x_{nn}\right\|_{v},1\right\} =1,\text{ }\forall v\notin S,$$
  hence the value of $H_\mathrm{aff}(\rho(\gamma))$ does not change if one only takes product over $v\in S$.

  If $v \in V^K_{\infty}$, we have
		\begin{equation}\label{Archimdean}
		\max \left\{\left\|x{}_{11}\right\|_{v}, \ldots,\left\|x{}_{nn}\right\|_{v},1\right\} \leq \left(|x{}_{11}|_{v}^{2}+\cdots+|x{}_{nn}|_{v}^{2}\right)^{[K_v\colon \R]  / 2} \leq n^{[K_v\colon\R]} \cdot \max \left\{\left\|x{}_{11}\right\|_{v}, \ldots,\left\|x{}_{nn}\right\|_{v},1\right\},
		\end{equation}
		where the second inequality is clear. For the first inequality, notice that the Hadamard's inequality on the determinant \cite[Exercise 9, p. 542]{Zorichbook} together with the arithmetic-geometric mean inequality yield
  $$1=[\det (\mathbf{x})]^2\leq \prod _{j=1}^n \left( \sum_{i=1}^n |x_{ij}|^2_v \right) \leq \left( \frac{\sum_{1\leq i,j\leq n}|x_{ij}|_v^2}{n}  \right)^n$$

		$$\Rightarrow 1 < \sqrt{n} \leq \left(\left|x_{11}\right|_{v}^{2}+\cdots+ \left|x_{1n}\right|_{v}^{2}+\cdots+\left|x_{n1}\right|_{v}^{2}+\cdots+ \left|x_{nn}\right|_{v}^{2}\right)^{1 / 2}.$$
  For this reason, 
\begin{align*}
    \max \left\{\left\|x{}_{11}\right\|_{v}, \ldots,\left\|x{}_{nn}\right\|_{v},1\right\}& =\left( \max \{ |x_{11}|_v,\dots, |x_{nn}|_v,1  \}\right)^{[K_v\colon \mathbb{R}]}\\
    &\leq \left( \max \{ \sqrt{|x_{11}|_v^2+\cdots +|x_{nn}|_v^2},1  \}\right)^{[K_v\colon \mathbb{R}]}\\
    &\leq \left( \sqrt{|x_{11}|_v^2+\cdots +|x_{nn}|_v^2}  \right)^{[K_v\colon \mathbb{R}]},
\end{align*}
and hence the first inequality of (\ref{Archimdean}) follows.

 For $v\in S \cap V^K_f$, we claim that 
  \begin{equation}\label{NonArchimdean}
		\max \left\{\left\|x{}_{1}\right\|_{v}, \ldots,\left\|x{}_{n^2}\right\|_{v}\right\} = \max \left\{\left\|x{}_{1}\right\|_{v}, \ldots,\left\|x{}_{n^2}\right\|_{v},1\right\}.
		\end{equation}
  To see this, notice that if $\mathbf{x}\notin \cO_v^{n^2}$, then there is a component with $\|x_i\|_v>1$, making (\ref{NonArchimdean}) obvious; if $\mathbf{x}\in \cO_v^{n^2}$, using ultrametric inequality to $\det \mathbf{x}=1$, we obtain that there is an entry $x_i$ of $\mathbf{x}$ with $\|x_i\|_v=1$, which yields (\ref{NonArchimdean}).

		The result follows by multiplying \eqref{Archimdean} and \eqref{NonArchimdean} over all $v \in S$. 
		
	\end{proof}
	
	Combining Theorem \ref{GNcounting} and Lemma \ref{HtildacompatiblewithH}, we obtain that
	\begin{equation}\label{Eq2}
	\#\{\gamma \in \Gamma \mid {H}_{\mathrm{aff}}(\rho(\gamma)) \leq T\} \geq \#\{\gamma \in \Gamma \mid \widetilde{H}(\gamma)  \leq T\} \sim \mu(B_{T}), \text{ as } T \to \infty.
	\end{equation}
	We aim at showing that the left hand side of (\ref{Eq2}) is $\gg  T^{\delta}$ for some $\delta >0$. To do so, it is enough to prove that $\mu(B_{T}) \gg  T^{\delta}$. For $v \in V^K$, let $\mu_v$ be the normalized Haar measure on $G(K_v)$ such that $\mu=\prod_{v \in S} \mu_v$. For $R >0$, define the $v$-local ball as 
 \[
	B_{v}(R) := \{g \in G(K_v) \mid \widetilde{H}_v(\rho(\gamma)) \leq R\}.
	\]
	Then $B_T$ contains the product of local balls $\prod_{v \in S} B_{v}\left(T^{[K\colon \mathbb{Q}]/s}\right)$, where $s= \#S$. In particular,
	\begin{equation}\label{Eq3}
	\mu(B_T) \geq \prod_{v \in S} \mu_v\left(B_{v}\left(T^{[K\colon \mathbb{Q}]/s}\right)\right).
	\end{equation}
	Thus to obtain the sought lower bound $\mu(B_T)\gg  T^{\delta}$ it is enough to establish the relation $\mu_v(B_{v}(T))\gg  T^{\delta'}$ for {\it one} $v \in S$. 
	
	The following key lemma, which is of independent interest, gives us what we seek.
	\begin{lemma}\label{lemma}
		Let $v \in V^K$ be a valuation such that $G(K_v)$ is non-compact; then there exists $\delta' >0$ such that
		\[
		\mu_v(B_{v}(T))\gg  T^{\delta'}\text{ as }T\to \infty.
		\]
		
	\end{lemma}
	\begin{proof}
		If $v$ is archimedean, $G(K_v)$ becomes a connected non-compact semi-simple real Lie group, our assertion is true by the asymptotic formula $\mu_v(B_v(T))\sim c' T^{\alpha}(\log T)^{\beta}$ ($c',\alpha>0,\beta \geq 0$) as $T\to \infty$, cf.  \cite[Theorem 2.7]{GW07} or \cite[Corollary 1.1]{Maucourant07}. So, let us assume that $v$ is non-archimedean. Since $G$ is $K_v$-isotropic, there exists a $K_v$-defined embedding $\iota\colon \Ga \hookrightarrow G_{K_v}$ (write  $\Ga \subseteq G_{K_v}$ if no risk of confusion), cf. \cite[Sec.\ 2.1.14]{PR}. Write $\rho_a \coloneqq\rho|_{_{\Ga}}=\rho\circ \iota$, which is non-constant.
		
		More explicitly, there exist polynomials $a_{ij}(x)\in K_v[x], 1 \leq i,j \leq n,$ not all constants, such that $\rho_a(x)=(a_{ij}(x))_{i,j}$. Then there exists a $C>0$ such that $\widetilde{H}_v(\rho_a(x))  \leq C \cdot ( \|x\|_v^d +1)$ for all $x \in K_v,$ where $d=\max\limits_{1\leq i,j\leq n}\deg a_{ij}(x)\geq 1$ and $\widetilde{H}_v$ is defined as (\ref{deflocalht}).

		Let $G(\cO_v) \coloneqq G(K_v) \cap \rho^{-1}(\SL_n(\cO_v))$, and $\Ga(\cO_v)'\coloneqq \Ga \cap G(\cO_v)$. Although $\Ga(\cO_v)'$ is \underline{not} necessarily equal to $\Ga(\cO_v)=\cO_v$, being a bounded subgroup of $\Ga(K_v)=K_v$, $\Ga(\cO_v)'$ is contained in $\pi_v^{-\ell}\cO_v$ for some $\ell \in \NN$.

		In order to obtain the sought lower bound, note that the following union is contained in $B_{v}(T)$,
		\[
		\bigcup_{\substack{x \in K_v\\ \widetilde{H}_v(\rho_a(x)) \leq T}}\iota(x) \cdot G(\cO_v),
		\]
		(this is because $\widetilde{H}_v(MN) \leq \widetilde{H}_v(M)\widetilde{H}_v(N)$ for every $N,M \in \GL_n(K_v)$ and $\widetilde{H}_v(g)=1$ for every $g \in \SL_n(\cO_v)$)
		and that the following union, contained in the one above, is disjoint,
		\begin{equation}\label{Disjoint}
		\bigsqcup_{\substack{[x] \in \mathcal{R} \\ \widetilde{H}_v(\rho_a(x)) \leq T}}\iota(x) \cdot G(\cO_v),
		\end{equation}
		where $\mathcal{R}$ is a set of representatives for the quotient $K_v/(\Ga(\cO_v)')$. Note that, as $T \to \infty$, we have
		\begin{multline*}  
		\#\{[x] \in \mathcal{R} \colon \widetilde{H}_v(\rho_a(x)) \leq T\} \geq  \#\{[x] \in \mathcal{R} \colon C \cdot ( \|x\|_v^d +1) \leq T\}  \\ =  \#\left\{[x] \in \mathcal{R} \colon \|x\|_v \leq (T/C-1)^{1/d}\right\}  \gg   \#\left\{[x] \in K_v/(\pi_v^{-\ell}\cO_v) \colon \|x\|_v \leq (T/C-1)^{1/d}\right\}   \asymp T^{1/d}.
		\end{multline*}
		In particular, the union \eqref{Disjoint} has (left) Haar measure $\gg  \mu_v(G(\cO_v))\cdot  T^{1/d}$, and this proves the sought bound.
	\end{proof}
	
	Summarizing, putting together Equations \eqref{Eq2} (recall that we proved this via Gorodnick-Nevo's result), \eqref{Eq3} and Lemma \ref{lemma}, we just proved:
	
	\begin{prop}[Counting property under ($\mathcal{C}_1$) and ($\mathcal{C}_2$)]\label{volestimatesimple}
		Let $K$ be a number field and let $S$ be a finite set of places of $K$ containing all archimedean ones. Let $G$ be a $K$-simple simply connected linear algebraic group and $\rho:G \to \SL_{n,K}, n \geq 2$ be a representation with finite kernel. Assume that $G(\cO_S)$ is infinite; then there exists $\delta >0$ such that
		\[
		\#\{\gamma \in G(\cO_S) \mid {H}_{\mathrm{aff}}(\rho(\gamma)) \leq T\} \gg  T^{\delta}\text{ as }T\to \infty.
		\]
	\end{prop}

   In other words, this completes the proof of (a)$\Longrightarrow$(b) part of Theorem \ref{volestimate} when $G,\Gamma$ satisfy both ($\mathcal{C}_1$) and ($\mathcal{C}_2$). It remains to explain how to deduce the general case to this one as well as (b)$\Longrightarrow$(a).

    \begin{proof}[Proof sketch of the deduction of Theorem \ref{volestimate} from Proposition \ref{volestimatesimple}] Let $K,S,G,\Gamma$ be as in the context of Theorem \ref{volestimate}. If $G$ is not reductive, then its unipotent radical, hence $G$ itself, contains a $K$-subgroup which is isomorphic to $\mathbb{G}_a$. Using arguments similar to the proof of Lemma \ref{lemma}, one can show that $\#\{P\colon H_{\mathrm{aff}}(P)<T\}\geq T^{\delta'}$ for some $\delta'>0$. Henceforth, we assume $G$ is \underline{reductive}.

    \vskip2mm
	\noindent
	{\bf Reduction to }($\mathcal{C}_2$): Since $\Gamma\subset G(K)$ is an $S$-arithmetic subgroup of $G$, we see that $\Gamma$ is commensurable with $\Gamma_0 \coloneqq G(\cO_S)$. The reduction follows from the next easy lemma.
 \begin{lemma}\label{lemmacomm} Let $\Delta_1,\Delta_2$ be two commensurable subgroups of $G(K)$, then
     \begin{itemize}
     \item[(i)] The group $\Delta^{\mathrm{der}}_1$ is infinite if and only if $\Delta^{\mathrm{der}}_2$ is;
     \item[(ii)] There exist $0<b_1<b_2$ such that 
     $$b_1\cdot \#\{P\in \Delta_1\colon H_{\mathrm{aff}}(P)<T\}\leq \#\{P\in \Delta_2\colon H_{\mathrm{aff}}(P)<T\}<b_2 \cdot  \#\{P\in \Delta_1\colon H_{\mathrm{aff}}(P)<T\}.$$
 \end{itemize}
 \end{lemma}

    The verification of both (i) and (ii) needs the usage of a finite set of representatives of each $\Delta_i/ \Delta_1\cap \Delta_2,i=1,2$; for (2), one also resorts to the height inequality $H_{\mathrm{aff}}(\gamma) \ll_{\theta
    } H_{\mathrm{aff}}(\theta \cdot \gamma)$ (cf. Lemma \ref{Lem:heightdifference}).

    So, we will assume from now on that $\Gamma=G(\cO_S).$
  \vskip2mm
	\noindent
	{\bf Reduction to }($\mathcal{C}_1$): Recall the following two helpful facts.
	
	\vskip2mm
	\noindent
	{\bf Fact A.} Let $\phi\colon G' \to G$ be an epimorphism of $K$-algebraic groups, and $\Gamma <G'(K)$ be a subgroup; then $\phi(\Gamma)$ is an $S$-arithmetic subgroup of $G$ if $\Gamma$ is an $S$-arithmetic subgroup of $G'$, cf. \cite[Theorem 5.9]{PR}. Moreover, if $\phi$ is an isogeny, then $\Gamma \to \phi(\Gamma)$ is a virtual isomorphism, i.e.\  the kernel is finite.
	\vskip2mm
	\noindent
	{\bf Fact B.} Let $G/K$ be reductive, we have exact sequences $1 \to Z \to G \to G^{\mathrm{ad}} \to 1$ and $1 \to G^{\mathrm{der}} \to G \to G^{\mathrm{ab}} \to 1$ of $K$-groups, where $Z$ and $G^{\mathrm{ab}}$ are $K$-tori, $G^{\mathrm{ad}}$ and $G^{\mathrm{der}}$ are semi-simple. In addition, the induced morphisms $G^{\mathrm{der}} \to G^{\mathrm{ad}}$ and $Z \to G^{\mathrm{ab}}$ are isogenies.
	\vskip2mm

     One also notices that $G^{\mathrm{der}}(\cO_S)$ is finite if and only if $G(\cO_S)^{\mathrm{der}}$ is.
	
	(b)$\Longrightarrow$(a) (recall that (a) and (b) are those in the statement of Theorem \ref{volestimate}). We prove the contrapositive: suppose $G^{\mathrm{der}}(\cO_S)$ is finite, then it is an easy exercise using Facts A and B above to prove that the image of $Z(\cO_S) \to G(\cO_S)$ is of finite index. Since $Z(\cO_S)$ is finitely generated, $G(\cO_S)$ is virtually multiplicative and finitely generated. Thus, over a finite field extension $L$ of $K$, we may then virtually diagonalize $G(\cO_S)$. In this way $G(\cO_S)\subset G(L)$ contains a (PEP) subgroup with finite index, so $\#\{P\in G(\cO_S)\colon H_{\mathrm{aff}}(P)<T\}< c\cdot (\log T)^{r'}$ by Theorem \ref{firstmainthm} and Lemma \ref{lemmacomm} (ii) \footnote{One may also use Lemma \ref{Lemrankequalrank} in the Appendix. We refer the reader to that lemma for more details.}, in particular the latter is slower than $T^{\delta}$ for any $\delta >0.$ 
	
	(a)$\Longrightarrow$(b): If $G^{\mathrm{der}}(\cO_S)$ is infinite, then we only need to treat the case $G=G^{\mathrm{der}}$. Let $\pi\colon G^{\mathrm{sc}}\to G^{\mathrm{der}}$ be the universal cover (which is an isogeny) and $G^{\mathrm{sc}}=G_1\times \cdots \times G_r$  be its decomposition into $K$-simple simply connected factors, cf. \cite[Theorem 2.6 (2)]{PR}.

 By Fact A, $G^{\mathrm{sc}}(\cO_S)=G_1(\cO_S)\times \cdots \times G_r(\cO_S)$ has finite index image in $G^{\mathrm{der}}(\mathcal{O}_S)$ under the universal cover $\pi$, hence $G^{\mathrm{sc}}(\cO_S)$ is infinite, so is one of its factors $G_i(\cO_S)$. To complete the reduction of Theorem \ref{volestimate} to ($\mathcal{C}_1$), it suffices to apply Proposition \ref{volestimatesimple} to $G_i$ and the finite-kernel representation $\rho_i\colon G_i \to G^{\mathrm{der}} \hookrightarrow \mathrm{SL}_{n,K}$.
    \end{proof}
    \begin{rema}[A classical approach to counting lattice points]\label{Rmk: classical} {\rm 
    To deduce Proposition \ref{volestimatesimple} from the volume bound Lemma \ref{lemma}, we used the asymptotic $\#\{\gamma \in G(\cO_S) \mid \widetilde{H}(\rho(\gamma)) \leq T\}\sim \mu(B_T)$ (with the notation of the beginning of this section), proved in modern literature such as \cite{GN12}. We would like to emphasize that this is not necessary for our purposes, as a lower bound $\gg \mu(B_T)^{\delta'}$ for some $\delta'>0$ (in place of $\sim \mu(B_T)$) would be sufficient. For $K$-isotropic groups it is much easier to prove Proposition \ref{volestimatesimple}, using a $K$-embedding $\mathbb{G}_a \hookrightarrow G$ (for the existence of such an embedding see \cite[Sec.\ 2.1.14]{PR}) and the fact that $S$-integral points on  $\mathbb{G}_a$ already grow polynomially in terms of the height. So, we focus on the case where $G\hookrightarrow \mathrm{SL}_{n,K}$ is a semi-simple $K$-{\bf anisotropic} group.
    
    In this context, the sought lower bound may be obtained using the more classical finiteness of the Tamagawa number $\tau(G) :=\mathrm{vol}(G(\mathbb{A}_K)\slash G(K))$ (cf. Borel and Harish-Chandra's fundamental paper \cite{B-HC62} for the proof in full generality, as well as some important earlier approaches such as Siegel \cite{Siegel35,Siegel36,Siegel37}, Weil \cite{Weil62} whose work was motivated by the study of quadratic forms). This may be done via a Minkowski's first convex body theorem type argument, that we sketch here (only for $G$ semi-simple and $\Gamma=G(\cO_S)$, to which one can easily reduce to). 
    
    For each $T>0$, define 
    $$D_T := \prod_{v\in S} B_v\left(T^{[K\colon \mathbb{Q}]/|S|}\right) \subseteq B_T,$$
    where $B_v(T):=\{ g \in G(K_v) \mid \widetilde{H}_v(g) \leq T\}$ and $B_T:=\{ g \in G_S = \prod_{v \in S} G(K_v) \mid \widetilde{H}(g) \leq T\}$ as in the beginning of this section.
    
    It is easy to verify that $D_T\cdot D_T^{-1}\subset B_{aT^n}$, with $a=n!$.
    
    The finiteness of the Tamagawa number $\tau(G)$ implies that the volume $\mathrm{vol}(G_S\slash G(\cO_S))$ is also finite (see e.g.\ \ the proof of \cite[Theorem 5.7]{PR}).

    Note now that the local covering $D_T \to G_S\slash G(\cO_S)$ must have a fiber $\{\alpha_1,\ldots,\alpha_f\}$ of cardinality $f\geq \mathrm{vol}(D_T)/\mathrm{vol}(G_S\slash G(\cO_S)) \gg \mathrm{vol}(D_T),$ and 
    \[
    \{I, \alpha_2\cdot \alpha_1^{-1},\ldots, \alpha_f\cdot \alpha_1^{-1}\} \subset G(\cO_S) \cap (D_T\cdot D_T^{-1})\subset G(\cO_S) \cap  B_{aT^n}.
    \]
    Thus $\# \{g\in G(\cO_S);\widetilde{H}(g)\leq aT^n\} \gg \mathrm{vol}(D_T)$, and hence $\#\{g\in G(\cO_S)\colon \widetilde{H}(g)\leq T\} \gg \mathrm{vol}\left(D_{(T/a)^{1/n}}\right)$. By Lemma \ref{lemma}, the latter is $\gg T^{\delta}$, for some $\delta>0$. Combining with Lemma \ref{HtildacompatiblewithH}, we thus obtain that $\#\{g\in G(\cO_S)\colon {H}_{\mathrm{aff}}(g)\leq T\} \gg \#\{g\in G(\cO_S)\colon \widetilde{H}(g)\leq T\} \gg T^{\delta}$, as wished.
  }
    \end{rema}

    We give in the following remark a second approach to proving $(a) \Rightarrow (b)$ in Theorem \ref{volestimate} when $G$ is reductive. It has the disadavantage of not being derived from very precise asymptotic estimates as Theorem \ref{GNcounting}, but has the advantage of being easily generalizable to more general linear groups (see more on this in Remark \ref{Rmk: T to the delta appendix}). 

    \begin{rema}\label{Rmk: a second approach}{\rm   
    Let $G< \mathrm{GL}_n$ be reductive, and $\Gamma < G$ be $S$-arithmetic with $\# \Gamma^{\mathrm{der}}=\infty$. The group $\Gamma$ is non-virtually-solvable (for otherwise its Zariski-closure, which contains the semi-simple group $G^{\mathrm{der}}$, would be virtually solvable). Then Tits' alternative \cite{Tits} asserts that $\Gamma$ contains a free subgroup $\Gamma_1:=\langle \gamma_1,\gamma_2 \rangle$ of rank $2$.  Let $W(l)$ denote the set of words in $\gamma_1,\gamma_2$ of length $\leq l$. Note that $\#W(l) \geq 2^l$. Letting $C:=\max (H_{\mathrm{aff}}(\gamma_1),H_{\mathrm{aff}}(\gamma_2))$, one has $H_{\mathrm{aff}}(\gamma) \leq n^{l-1}C^l\leq (nC)^l$ for all $\gamma \in W(l)$. Hence:
  \begin{equation}\label{Eq: ineq with words}
       \#\left\{\gamma \in \Gamma \colon H_{\mathrm{aff}}(\gamma) \leq T\right\} \geq \#W\left( \log_{nC} (T) \right) \gg T^\delta, \ \delta =\log_{nC} (2),
 \end{equation}
(re)proving Theorem \ref{volestimate} for $\Gamma$.
   The theme of exponential growth in linear groups (implicit in our argument above) has been largely explored in the literature. For this and related topics see, for instance, \cite{B}, \cite{BCLM}, \cite{EMO} or \cite[Section 5.B]{Gro}. } 
   \end{rema}
    
	\section{Proof of the main theorems}\label{Sec.MainProofs}
	In this section, we are going to prove the main theorems announced in the introduction. We first treat the main quantitative result. 
		
	\begin{proof}[Proof of Theorem \ref{firstmainthm}]  By Proposition \ref{decomposepep}, there is a decomposition of $\Sigma$ such that $\Sigma= \Sigma_1\cup \cdots \cup \Sigma_l$, where $\Sigma_i=
		\mathbf{f}_i(\ZZ^{r_i})$ with $\mathbf{f}_i$ be reduced (PEP) and $r \geq r_i \geq 0$ for each $i$.  After re-enumerating those $\Sigma_i$, we may assume that 
		$$r_{\mathrm{max}}\coloneqq r_1=r_2=\cdots=r_{l'}>r_{l'+1}\geq \cdots \geq r_l.$$ 

We prove the result by showing that $r'$ in the statement should be taken to be $r_{\mathrm{max}}$. The verification goes by induction on $r_{\mathrm{max}}$. If $r_{\mathrm{max}}=0$, then $\Sigma$ is a finite set, our assertion is evident.  Now let $r_{\mathrm{max}}>0$.
  
		 Let $\varphi: \ZZ^{r_1} \sqcup \ZZ^{r_2} \sqcup \cdots \sqcup \ZZ^{r_{l'}} \to K^s$ be the unique map such that $\varphi|_{\ZZ^{r_i}}=\mathbf{f}_i$ for each $i$.
		Notice that
		\begin{equation}\label{Mainheight0}
			\#\{\gamma \in \Sigma_{l'+1}  \cup \cdots \cup \Sigma_l\mid h_{\mathrm{aff}}(\gamma) \leq T \} = O(T^{r_{\mathrm{max}}-1})
		\end{equation}
		by induction hypothesis.
		
		Applying Theorem \ref{fiberpep} to each pair $(\mathbf{f}_i,\mathbf{f}_j), i,j=1,\ldots,l'$, we know that there exists a set $D_{i,j} \subseteq \ZZ^{r_i}=\ZZ^{r_{\mathrm{max}}}$, which is a union of finitely many cosets of rank $\leq r_i-1=r_{\mathrm{max}}-1$, and a natural number $N_{i,j}$ such that the  function
		\[
		\ZZ^{r_{\mathrm{max}}} \to \ZZ_{\geq 0} \cup \{\infty\}  \ \ \u{n} \mapsto \#\mathbf{f}_j^{-1}(\mathbf{f}_i(\u{n}))
		\]
		has finite values and depends only on $\u{n} \bmod N_{i,j}$ for $\u{n} \notin D_{i,j}$. As a consequence, the function $\ZZ^{r_{\mathrm{max}}} \to \NN \cup \{\infty\}, \ \u{n} \mapsto \#\varphi^{-1}(\mathbf{f}_i(\u{n}))$, which is the sum of the $l'$  functions above ($i$ fixed, and $j=1,\ldots, l'$), attains finite values depending only on $\u{n}$ modulo $N\coloneqq\mathrm{L.C.M} _{1\leq i,j\leq l'}N_{i,j}$ for all $\u{n} \notin D_i\coloneqq\bigcup_{j=1}^{l'} D_{i,j}$. 
  
		Letting $\Sigma' =\bigcup_i \mathbf{f}_i(D_i)$, we have the identity  
		\begin{equation}\label{Mainheight1}
			\#\{\gamma \in \Sigma_1\cup \cdots\cup \Sigma_{l'} \mid h_{\mathrm{aff}}(\gamma) \leq T \} = \sum_{i=1}^{l'} \sum_{\substack{\u{n}\in \ZZ^{r_{\mathrm{max}}} \backslash D_i\\ h_{\mathrm{aff}}(\mathbf{f}_i(\u{n})) \leq T }}\frac{1}{\# \varphi^{-1}(\mathbf{f}_i(\u{n}))} + O(\#\{\gamma \in \Sigma' \mid h_{\mathrm{aff}}(\mathbf{f}_i(\u{n})) \leq T \}).
		\end{equation}
		Again, the error term in (\ref{Mainheight1}) is $O(T^{r_{\mathrm{max}}-1})$ by induction hypothesis. On the other hand, for each $i$, the Zero Locus Theorem \ref{zerolocuspep} yields that the set
  $$B_i\coloneqq\{\mathbf{n}\in \ZZ^{r_{\mathrm{max}}}\colon \mathbf{f}_i(\mathbf{n})\text{ is degenerate}\}$$
  is equal to a finite union of cosets of rank $<r_{\mathrm{max}}$. Therefore, according to Proposition \ref{Prop:norm}, we have
\begin{align}\label{Mainheight2}
    \sum_{\substack{\u{n}\in \ZZ^{r_{\mathrm{max}}} \backslash D_i \\ h_{\mathrm{aff}}(\mathbf{f}_i(\u{n})) \leq T }}\frac{1}{\# \varphi^{-1}(\mathbf{f}_i(\u{n}))} &=\sum_{\mathbf{a} \in \{0,\ldots, N-1\}^{r_{\mathrm{max}}}} \left[\frac{1}{\# \varphi^{-1}(\mathbf{f}_i(\u{n}_{\mathbf{a}}))} \cdot \#\{\u{n} \in (\mathbf{a}+N\ZZ^{r_{\mathrm{max}}})\backslash D_i \mid h_{\mathrm{aff}}(\mathbf{f}_i(\u{n})) \leq T\}\right]\notag \\
    &\sim \sum_{\mathbf{a} \in \{0,\ldots, N-1\}^{r_{\mathrm{max}}}} \left[\frac{1}{\# \varphi^{-1}(\mathbf{f}_i(\u{n}_{\mathbf{a}}))} \cdot \#\{\u{n} \in (\mathbf{a}+N\ZZ^{r_{\mathrm{max}}})\backslash (D_i\cup B_i) \mid h_{\mathrm{aff}}(\mathbf{f}_i(\u{n})) \leq T\}\right]\notag \\
    &\sim \sum_{\mathbf{a} \in \{0,\ldots, N-1\}^{r_{\mathrm{max}}}}\left[\frac{1}{\# \varphi^{-1}(\mathbf{f}_i(\u{n}_{\u{a}}))}\#\{\u{n} \in (\mathbf{a}+N\ZZ^{r_{\mathrm{max}}} )\backslash (D_i\cup B_i) \mid \|\u{n}\|_{\mathbf{f}_i} \leq T\} \right]\notag\\& \sim  \sum_{\mathbf{a} \in \{0,\ldots, N-1\}^{r_{\mathrm{max}}}} \frac{1}{\# \varphi^{-1}(\mathbf{f}_i(\u{n}_{\u{a}}))}\cdot \frac{1}{N^{r_{\mathrm{max}}}}  c_i  T^{r_{\mathrm{max}}}\notag \\
 &=c'  T^{r_{\mathrm{max}}},
\end{align}
where $\norm{\cdot}_{\mathbf{f}_i}$ is the norm defined in Proposition \ref{Prop:norm2}, $c_i= \mathrm{vol}(\{\u{v} \in \R^{r_{\mathrm{max}}} \mid \|\u{v}\|_{\mathbf{f}_i} \leq 1\})>0$, and $\u{n}_{\u{a}}$ is any element in $(\u{a} + N\cdot \ZZ^{r_{\mathrm{max}}})\backslash \bigcup\limits _{i=1}^{l'} D_i$. Combining \eqref{Mainheight0}, \eqref{Mainheight1} and \eqref{Mainheight2}, the theorem is proven with $r'=r_{\mathrm{max}}$.

	\end{proof}

\begin{rema}\label{Rmk:meaningc}{\rm
    As it stems from the proof above,  the constant $c$ in Theorem \ref{firstmainthm} has a very clean geometric meaning, at least when $\Sigma$ is the image of a reduced vector of purely exponential polynomials. Indeed, in this case, $c$ is a rational multiple of the volume of the norm-ball
    \[
    \{v \in \R^r \mid \norm{v}_{\mathbf{f}} \leq 1\},
    \]
    where $\norm{*}_{\mathbf{f}}:\R^r \to \R_{\geq 0}$ is the height-norm associated to the monomials of $\mathbf{f}$, defined as in Proposition \ref{Prop:norm2}. 
    The rational coefficient is determined by the combinatorics of the fibers of $\mathbf{f}$. The height-norm $\norm{*}_{\mathbf{f}}:\R^r \to \R_{\geq 0}$ is by construction the maximum function of a set of linear forms whose coefficients are logarithms of algebraic numbers.

    When $\Sigma$ is non-reduced, then $c$ is a (rational) linear combination of volumes as above, where the coefficients still have a ``combinatorics of fibers'' meaning.}
\end{rema}

\vskip3mm

Now the main height inequality Theorem \ref{minspecial} follows easily from Corollary \ref{CoromainThm}.

    \begin{proof}[Proof of Theorem \ref{minspecial}]
	This follows directly by applying Corollary \ref{CoromainThm}. 
    Indeed, let $b \in \mathbf{f}(\ZZ^r)$, and let $\u{n} \in \mathbf{f}^{-1}(b)$ be an element whose norm $\|\u{n}\|_{\infty}$ is minimal. If $\u{n} \in L_i$ (with the notation of Theorems \ref{mainlongpaper} and its corollary \ref{CoromainThm}), note that $\min_{\u{v} \in \mathcal{K}_i} \| \u{n} +\u{v} \|_{\infty}= \|\u{n}\|_{\infty}$ by the minimality of $\u{n}$. Hence
    we have:
    \[
        h_{\text {aff }}(\mathbf{f}(\mathbf{n})) \geq C \cdot \min_{\u{v} \in \mathcal{K}_i} \| \u{n} +\u{v} \|_{\infty} = C \cdot \| \u{n}\|_{\infty},
     \]
    except for $\u{n}$ lying in a finite union of translates of $\mathcal{K}_i$. The image of each translate consists of a single point, so this leads to only finitely many exceptional $b$, as wished.
    \end{proof}
	\vskip2mm
	Next we shift our attention to the main application of the above asymptotic result in group theory, i.e.\  Theorem \ref{secondmainthm}. Before proving it, we first need the asymptotic estimate given by Lemma \ref{Cor} in the introduction (see also \cite{CDRRZcr}).

	\begin{proof}[Proof of Lemma \ref{Cor}.] We first prove the lemma when $g=g_u \neq I_n$ is unipotent. Write $g_u=I_n+h$, where $h$ is a nonzero nilpotent matrix. Using Newton's binomial expansion for $g_u^l=(I_n+h)^l, l \in \ZZ \setminus \{0\}$, we see that $g_u^l=(c_{ij}(l))_{i,j}$, where $c_{ij}(X)\in K[X]$ are fixed polynomials of degree $<n$ for all $i,j$. Since $h$ is nonzero, not all $c_{ij}(X)$ are constants. 

	Take a non-constant $c_{ij}(X)$ with $\deg c_{ij}(X)=k\geq 1$. We obtain from elementary properties of (Weil) heights including $H_{\mathrm{aff}}(\alpha \beta)\leq H_{\mathrm{aff}}(\alpha) H_{\mathrm{aff}}(\beta) $ and $H_{\mathrm{aff}}(\alpha_1+\cdots +\alpha_m)\leq r H_{\mathrm{aff}}(\alpha_1)\cdots H_{\mathrm{aff}}(\alpha_m)$ that there exist $a,b>0$ such that
	$$H_{\mathrm{aff}}(c_{ij}(l))\leq a H_{\mathrm{aff}}(l)^b=a|l|^b.$$
	
	Note that if $l\in \ZZ, |l|\leq N$ and $g_u^l\in \Sigma$, then $H_{\mathrm{aff}}(c_{ij}(l))\leq a N^b$. 
 
    Observe that $c_{ij}(\ZZ)$ is a subset of $\Sigma_{ij}$, where the latter is the (PEP) set consisting of the $(i,j)$-entries of elements in $\Sigma$. Also, since $g_u$ is non-trivial unipotent, $g_u^l$ are distinct for $l\in \NN$. Thus
    \begin{multline*}
        \# \{l\in \ZZ\colon |l| \leq N, g_u^l\in \Sigma\}  \leq \# \{l\in \ZZ\colon  |l| \leq N, c_{ij}(l)\in \Sigma_{ij}\} \leq \frac{1}{k}\# \{P\in \Sigma_{ij}\colon H_{\mathrm{aff}}(P)\leq aN^b\}\\ =O\left((\log (aN^b))^{\mathrm{rk}_{\mathrm{PEP}}\Sigma_{ij}+1}\right),
    \end{multline*}
	where the latter comes from Theorem \ref{firstmainthm}. This  concludes the case $g=g_u$ since $r \geq \mathrm{rk}_{\mathrm{PEP}}\Sigma_{ij}$.
	
	In general, let $g=g_ug_s=g_sg_u$ be the Jordan decomposition, where $g_u$ is unipotent and $g_s$ is semi-simple (we may assume both $g_s,g_u$ have entries in $K$). Let $\Sigma_{\mathrm{new}}$ be the new (PEP) set defined by $\langle g_s\rangle \cdot\Sigma$; then 
	$$\#\{l\in \ZZ\colon |l| \leq N, g^l\in \Sigma\}\leq \#\{l\in \ZZ\colon |l| \leq N,  g_u^l\in \Sigma_{\mathrm{new}}\}=O((\log N)^r+1),$$
	as required.
	\end{proof}
	We also need the following lemma about specialization, cf. \cite[\S 2]{CRRZ}. 
		\begin{lemma}[Specialization]\label{specializationlemma}
  Let $R$ be a finitely generated $\QQ$-algebra without zero divisors. Given a non virtually solvable subgroup $\Delta \subset \mathrm{GL}_n(R)$ and a non-semi-simple element $\gamma \in \Gamma$, there exists a $\QQ$-algebra homomorphism $\theta\colon R\to K$ to a number field $K$ such that for the resulting group homomorphism $\Theta\colon \mathrm{GL}_n(R)\to \mathrm{GL}_n(K)$, the group $\Gamma'\coloneqq\Theta (\Gamma)$ is not virtually solvable and the matrix $\Theta(\gamma)$ is non-semi-simple.
  \end{lemma}

	The proof goes exactly on the same lines as the one of \cite[Proposition 2.1]{CRRZ}, except that the last condition 	``$\Theta(\gamma)$ is non-semi-simple'' above needs a little care. Now we prove Theorem \ref{secondmainthm}. 

	\begin{proof}[Proof of Theorem \ref{secondmainthm}] Let us reduce to the number field case. Assume that the three assertions in Theorem \ref{secondmainthm} are equivalent when $K$ is a number field. Let $\Gamma \subset \mathrm{GL}_n(K)$ ($K$ is an arbitrary field of characteristic zero) be a linear group which has (PEP), i.e.\  $\Gamma =\mathbf{f}(\ZZ^r)$ for some vector of purely exponential polynomials $\mathbf{f}$. Let $R$ be the $\QQ$-algebra of finite type generated by the coefficients and bases, with their inverses, of an expression for $\mathbf{f}$. Assume by contradiction that $\Gamma$ contains a non-semi-simple element $\gamma$. Let $\theta, \Theta$ be as in Lemma \ref{specializationlemma} above. Then, on the one hand, $\Theta(\gamma)$ is non-semi-simple, on the other hand, $\Gamma'=\Theta(\Gamma)$ is a (PEP) linear group over a number field, thus contains only semi-simple elements because we assumed the validity of Theorem \ref{secondmainthm} in the number field case. This contradiction implies that $\Gamma$ contains only semi-simple elements. Also, the number field case of Theorem \ref{secondmainthm} shows that $\Gamma'$ is virtually abelian, thus virtually solvable, therefore the fact above guarantees that $\Gamma$ is virtually solvable as well. Now arguing as the proof of Corollary 1.2 of \cite{CRRZ}, we obtain that $\Gamma$ is virtually abelian. But the Zariski-closure of an abelian linear group consisting of semi-simple elements is diagonalizable (group of multiplicative type), so the connected component $G^{\circ}$ of $G\coloneqq\overline{\Gamma}^{\mathrm{Zar}}$ is a torus. Recall the following group-theoretic proposition cf. \cite[Proposition 6.2]{CRRZ}.
		
		\begin{prop}
		    Let $K$ be a field of characteristic zero; then for any finitely generated ring $B\subset K$, a commutative subgroup $\Delta \subset \mathrm{GL}_n(B)$ consisting of semi-simple elements is finitely generated.
		\end{prop}
		Taking $B$ to be the subring of $K$ generated by the bases and coefficients of an expression of $\mathbf{f}$; then $\Gamma=\mathbf{f}(\ZZ^r)\subset \mathrm{GL}_n(B)$. Indeed, it is clear that for every $\gamma \in \Gamma, \gamma \in M_n(B)$, but also $\gamma^{-1} \in \Gamma \subseteq M_n(B)$ hence $\det (\gamma) \in B \cap B^{-1}= B^*$, and thus $\gamma \in \GL_n(B)$. Hence the abelian group $\Delta\coloneqq\Gamma \cap G^{\circ}\subset \mathrm{GL}_n(B)$ is finitely generated, and so is $\Gamma$ itself. This establishes (1)$\Rightarrow$(3). The equivalence of (2) and (3) was done in \cite{CRRZ}, and the implication (2)$\Rightarrow$(1) is evident. We have therefore reduced to the number field case.
		
		In the rest of the proof, it suffices to verify the theorem for the case when $K$ is a number field. Notice that our argument is (non-trivially) analogous to the one in Theorem 1.1 in \cite{CRRZ}.
		
		Suppose $\Gamma$ has (PEP), i.e.\  $\Gamma=\mathbf{f}(\mathbb{Z}^r)$ for some vector of purely exponential polynomials $\mathbf{f}$ in $r$ variables. 
		Then it follows immediately from Lemma \ref{Cor} that $\Gamma$ consists entirely of semi-simple elements. 
		
		Suppose $\Gamma$ is not virtually solvable, take $G\subset \mathrm{GL}_{n,K}$ be the Zariski closure of $\Gamma$. Let $R$ be the radical of the connected component $G^{\circ}$ of $G$ and consider the natural homomorphism $\varphi\colon G\to G\slash R=G'$. Then both $G'$ and $\varphi$ are defined over $K$ as well, we may take an embedding $G'\subset \mathrm{GL}_{n',K}$. Take natural number $N$ to be sufficiently large so that the map defined by the following ``twist''
		$$\widetilde{\varphi}\colon G\to G',\text{  }g\mapsto (\det g)^N \cdot \varphi(g)$$
		is given by polynomials. Then $\Gamma'\coloneqq\widetilde{\varphi}(\Gamma)$ is a linear group in $G'(K)$ which has purely exponential parametrization. Since $\Gamma$ is non-virtually solvable, the connected component of the Zariski closure $G'$ of $\Gamma'$ is a non-trivial semi-simple $K$-group times a $1$-dimensional torus. Thus in the rest of the proof, it suffices to treat the case where the connected component
		$G^{\circ}$ of $G=\overline{\Gamma}^{\mathrm{Zar}}$
		is a non-trivial semi-simple group times a $1$-dimensional torus. 
		
		Let $\Lambda$ be the multiplicative subgroup of $K^*$ generated by the bases $\lambda_i$ appearing in the components $\mathbf{f}$. By Proposition 3.5 in \cite{CRRZ} (recall that this relies on the theory of {\bf generic elements} developed by G. Prasad and the third author, cf. \cite{PrR1}\cite{PrR2}\cite{PrR3} etc.), there exists a matrix $g\in\Gamma$ admitting an eigenvalue $\lambda$ which is not a root of unity and such that $\langle\lambda\rangle \cap  \Lambda=\{1\}$. 
		
		Using Laurent's Theorem cf. \cite[Theorem 10.10.1]{EG15}, then an almost identical process as the one in the proof of the key matrix statement \cite[Theorem 4.1]{CRRZ} yields the following.
		\begin{prop} \label{keymatrixcrrz21}Let $\Sigma =\mathbf{f}(\mathbb{Z}^r) \subset \mathrm{GL}_n(\overline{\mathbb{Q}})$ be a (PEP) set under the purely exponential polynomial $\mathbf{f}=(f_1,\dots,f_{n^2})$ where
			$$f_i(\mathbf{x})=\sum_{j=1}^e a_j \lambda_1^{l_{1,j}(\mathbf{x})}\cdots \lambda_k^{l_{k,j}(\mathbf{x})}.$$
			Let $\Lambda$ be the multiplicative subgroup of $\overline{\mathbb{Q}}^{*}$ generated by all bases $\lambda_l$'s involved in the definition of $\mathbf{f}$. 
			Then for any semi-simple matrix $\gamma \in \mathrm{GL}_n(\overline{\mathbb{Q}})$ that has an eigenvalue $\lambda \in \overline{\mathbb{Q}}^{*}$ which is not a root of unity and for which $\langle \lambda \rangle \cap \Lambda = \{ 1 \}$, the intersection $\langle \gamma \rangle \cap \Sigma$ is finite. 
		\end{prop}

		It follows that only finitely many powers of $g$ are inside the image $\mathbf{f}(\mathbb{Z}^r)$, whence a contradiction. Thus $\Gamma$ is virtually solvable.
		
		Next, an argument similar to the one in the proof of Corollary 1.2 in \cite{CRRZ} guarantees that $\Gamma$ is virtually abelian.
		
		Now, since $\Gamma$ is virtually abelian and consisting only of semi-simple elements, $G^{\circ}$ must be a torus. Then, after a suitable conjugation, we may assume that the purely exponential parametrization of $\Gamma$ has shape
		\begin{equation*}
		g=\mathrm{diag}(f_1(\mathbf{x}),\ldots,f_n(\mathbf{x})),\qquad \mathbf{x}\in\mathbb{Z}^r,
		\end{equation*}
		for suitable purely exponential polynomials $f_1,\ldots,f_n$. Since $\Gamma$ is a group, the value sets $f_1(\mathbb{Z}^r),\dots,f_n(\mathbb{Z}^r)$ must all be multiplicative subgroups of $K^*$. Thus according to the Image Theorem \ref{imagetheorem}, $f_i(\ZZ^r)$ are all finitely generated multiplicative groups sit inside $K^*$. Therefore, $\Gamma$ is finitely generated, and even has (BG), as it is a finitely generated abelian group.
	\end{proof}
	
	\section{Effectivity issues}\label{noteffective}
	In this section, we discuss in detail the computational aspects of the main height asymptotic in \S \ref{pfht}. 
        In particular, we link the ineffectiveness of the number of exceptional cosets in Corollary \ref{CoromainThm} (and thus a fortiori on the unavailability of an error term in Theorem 
        \ref{mainlongpaper}) to the ineffectiveness of the height of solutions of the $S$-unit equation.

	Recall the following breakthrough, which was a surprising {\bf quantitative} result, established in Evertse, Schlickewei and Schmidt's distinguished article  \cite{effective02}.
	\begin{theo}\label{Sunitquan02} Let $K$ be a field of characteristic zero, let $n\geq 2$, let $a_1,\dots,a_n\in K^{*}$ and let $\Gamma$ be a subgroup of $(K^{*})^n$ of rank $r$. Then the number of non-degenerate solutions of
		\begin{equation}\label{suniteq} 
		a_1x_1+\cdots +a_nx_n=1 \text{ in }(x_1,\dots,x_n)\in \Gamma
		\end{equation}
		can be estimated from above by a quantity $A(n,r)$ depending only on $n$ and $r$. One may take $A(n,r)=e^{(6n)^3(r+1)}$.
	\end{theo}
	Relevant works and subsequent improvements have been done by various authors including Amoroso, Viada, R\'emond etc., cf. for example, \cite{effective09,effectiveremond}.

        We emphasize here that although we do have effective upper bounds for the {\bf number} of non-degenerate solutions of (\ref{suniteq}) \cite{Schmidt89,evertse96, SchlickeweiSchmidt02, EF13}, an effective {\bf height upper bound} of these solutions (when $K$ is a number field) seems to be far beyond reach at this moment.
	
	It is clear that the theorems in this paper are closely related with $S$-unit equations. However, in contrast with Theorem \ref{Sunitquan02}, the number of maximal exceptional cosets in Corollary \ref{CoromainThm} cannot be explicitly bounded in our framework. More precisely, we prove (see the \underline{Primary Computational Claim} below) that an effective bound on the number of maximal exceptional cosets would lead to an effective bound on the height of solutions of Theorem \ref{Sunitquan02}.

	We begin with an intuitive example.
	\vskip5mm
	\noindent
	{\bf Example.} {\rm Let $f(n)=a_1\lambda_1^n+\cdots +a_m\lambda_m^n$ (assume $\lambda_1,\dots,\lambda_m,2$ are multiplicatively independent) be a \underline{reduced} purely exponential polynomial in one variable, with bases and coefficients in a number field $K$. Consider the vector
	of purely exponential polynomials:
	$$\widetilde{\mathbf{f}}=\left(f,2^{(\cdot)}\right)\colon \mathbb{Z}^2\to K^2,\text{   }(n,l)\mapsto \left(f(n),2^l\right).$$
	Then $\widetilde{\mathbf{f}}$ is also reduced. Let $C=C_1'>0$ be the effective constant as in Corollary \ref{mainhtsteponecoro}. Choose any solution of (\ref{mainhtstepone}), i.e.\  $n\in \ZZ$ such that
	\begin{equation}\label{oldpep}
	f(n)\text{ is non-degenerate and  }h_{\mathrm{aff}}(f(n))\leq C\cdot |n|,
	\end{equation}
	holds. Then we get solutions $(n,1),(n,2),\dots, (n,\lfloor\log_2 (\varepsilon\cdot |n|)\rfloor)$ to the inequality 
	\begin{equation}\label{newpep}
	\widetilde{\mathbf{f}}(n,l)\text{ is non-degenerate and }h_{\mathrm{aff}}(\widetilde{\mathbf{f}}(n,l))\leq 2 C\cdot \max\{|n|,|l|\}, \ (n,l)\in \ZZ^2.
	\end{equation}

	Thus 
    \begin{equation}\label{bound}
	\# \{\text{solutions of }(\ref{newpep})\}\geq \log_2 \left(C\cdot \max\{|n|\colon n\text{ is a solution of }(\ref{oldpep})\}\right)-1.
	\end{equation}
	Therefore, if we had an effective bound for the {\bf number} of solutions of (\ref{newpep}), by \eqref{bound}, we would then obtain an effective bound for the {\bf size} of solutions of (\ref{oldpep}).}

	The following claim, which generalizes the idea of the example above, is the main contribution of this section. This claim essentially illustrates the difficulty of making the number of exceptional cosets in Corollary \ref{CoromainThm} to be effective.
	
	\vskip2mm
	\noindent
	\underline{\bf Primary Computational Claim.}  Suppose that for every vector of purely exponential polynomials $\mathbf{f}$, the \underline{number} of maximal exceptional cosets in Corollary \ref{CoromainThm} has an effective upper bound.
	Then, for every finite $S$ containing $V^K_{\infty}$, there would exist effective $\delta,B >0$ such that for every non-degenerate solution $\mathbf{x}=(x_1,\dots,x_r)\in (\cO_S^{*})^r$ to the inequality
	\begin{equation}\label{Eqnoneff}
		h_{\mathrm{aff}}(x_1+\cdots+x_r)\leq \delta \cdot \max\{h_{\mathrm{aff}}(x_1),\dots , h_{\mathrm{aff}}(x_r)\},
	\end{equation}
	we have $h_{\mathrm{aff}}(\mathbf{x})\leq B$.
	
	\begin{rema} {\rm In particular, under the (very strong) hypothesis of the above ``Primary Claim about Effectivity'', we would deduce that the height of the non-degenerate solutions of the $S$-unit equation (\ref{suniteq}) would be effectively boundable. The latter is, as mentioned above, a very difficult open problem, due to the presently unreachable effectiveness of the subspace theorem. }
	\end{rema}
	In the rest of this section we will explain how to obtain the Primary Computational Claim above.
		Note that an effective $\delta>0$ such that there are only finitely many solutions to 
		\eqref{Eqnoneff} exists by Corollary \ref{evertseinequality}. So, let us explain that we may obtain an effective bound on their height (under our strong assumption).
		
		Let $W_K \subseteq K^*$ be the finite group of roots of unity in $K$, and let $\lambda_1,\dots,\lambda_m$ be a basis for a ($\ZZ$-free) complement of $W_K$ in $\cO_S^*$, where $m$ is a finite number by Dirichlet's $S$-unit theorem. (Note that the $\lambda_i$'s are automatically {\bf multiplicatively independent}.)
		For each vector $\boldsymbol{\mu}=(\mu_1,\dots,\mu_r)\in W_K^r$, consider the following purely exponential polynomial $\mathbf{f}_{\boldsymbol{\mu}}\colon \ZZ^{rm}=(\ZZ^m)^r\to K$:
		$$\mathbf{f}_{\boldsymbol{\mu}}[(n_{11},\dots,n_{1m}),\dots,(n_{r1},\dots,n_{rm})]\coloneqq\mu_1 \lambda_1^{n_{11}}\cdots \lambda_m^{n_{1m}}+\cdots+\mu_r \lambda_1^{n_{r1}}\cdots \lambda_m^{n_{rm}}.$$
		Choose a positive integer $\alpha$ such that $\alpha, \lambda_1,\dots,\lambda_m$ are still multiplicatively independent. For each $\boldsymbol{\mu}=(\mu_1,\dots,\mu_r)\in W_K^r$, we further define a new vector of purely exponential polynomials $\widetilde{\mathbf{f}}_{\boldsymbol{\mu}}\colon \ZZ^{rm+1}=\ZZ^{rm}\times \ZZ\to K$ by $\widetilde{\mathbf{f}}_{\boldsymbol{\mu}}(\mathbf{n},l)=(\mathbf{f}_{\boldsymbol{\mu}}(\mathbf{n}), \alpha ^l).$ 
		
		Then it is easy to see that both $\mathbf{f}_{\boldsymbol{\mu}}$ and $\widetilde{\mathbf{f}}_{\boldsymbol{\mu}}$ are reduced (PEP). So in our case, all information in Corollary \ref{CoromainThm} we will need is included in Corollary \ref{mainhtsteponecoro}.
		
		Again, by Dirichlet's $S$-unit theorem, combined with the fact that $\lambda_1,\dots,\lambda_m$ are multiplicatively independent, we obtain that for every non-degenerate solution $(x_1,\dots,x_r)$ of (\ref{Eqnoneff}), there exist unique $\boldsymbol{\mu}=(\mu_1,\dots,\mu_r)\in W_K^r$ and $\mathbf{n}=(n_{11},\dots,n_{1m})\times\cdots \times (n_{r1},\dots, n_{rm})\in \ZZ^{rm}$ such that
		$$(x_1,\dots,x_r)=(\mu_1\lambda_1^{n_{11}}\cdots \lambda_m^{n_{1m}},\dots, \mu_r\lambda_r^{n_{r1}}\cdots \lambda_m^{n_{rm}}).$$
		Keeping in mind Proposition \ref{Prop:norm2}, this $\u{n}$ gives a non-degenerate solution (which means making $h_{\mathrm{aff}}({f}_{\boldsymbol{\mu}}(\mathbf{n}))$ to be a non-degenerate sum and such that the inequality below holds) of: 
		\begin{equation}\label{oldpep1} 
		\refstepcounter{equation}
		\tag*{(\theequation)\ensuremath{_{\boldsymbol{\mu}}}}
		h_{\mathrm{aff}}({f}_{\boldsymbol{\mu}}(\mathbf{n}))\leq \delta' \cdot \| \mathbf{n}\|_{\infty},
		\end{equation}
		for some $\delta'=\delta \cdot B$, where $B>0$ is an explicitly computable constant that depends only on $\lambda_1, \ldots, \lambda_m$ (in particular, not depend on the non-degenerate solution $(x_1,\dots,x_r)$), in other words, $\delta'$ is also effective. Now, arguing as in the previous example, for any non-degenerate solution $\mathbf{n}\in \ZZ^{rm}$ of \ref{oldpep1}, we get non-degenerate solutions $(\mathbf{n},1),(\mathbf{n},2),\dots,(\mathbf{n}, \lfloor\log_{\alpha}( \delta \cdot \| \mathbf{n}\|_{\infty})\rfloor)$ to the inequality
		\begin{equation}\label{newpep1}
		h_{\mathrm{aff}}(\widetilde{\mathbf{f}}_{\boldsymbol{\mu}}(\mathbf{n},l))\leq 2\delta'\cdot\|(\mathbf{n},l)\|_{\infty} =2\delta'\cdot \max \{\| \mathbf{n}\|_{\infty} ,|l|\}.
		\end{equation}

		Thus the number of non-degenerate solutions of (\ref{newpep1}) is at least $\lfloor\log_{\alpha}( \delta \cdot  R_{\boldsymbol{\mu}}) \rfloor$, where $R_{\boldsymbol{\mu}}$ is defined to be the maximum of $\| \mathbf{n}\|_{\infty}$ as $\mathbf{n}$ varies among the non-degenerate solutions of \ref{oldpep1}. Replacing $\delta$ by $\min \{\frac{C}{3B},\delta\}$ when necessary, we may assume that $2 \delta' < C$, where $C=C_1'$ is the effective constant appearing in Corollary \ref{mainhtsteponecoro}. So, if we have an effective upper bound for the number of non-degenerate solutions of (\ref{newpep1}), we will have an effective upper bound for the $\infty$-norm of non-degenerate solutions of \ref{oldpep1} for all $\boldsymbol{\mu}\in W_K^r$. Applying Proposition \ref{Prop:norm2} again, we obtain an effective height bound on the non-degenerate solutions of (\ref{Eqnoneff}).

	\section{An open question}\label{Par.Open}
 
	In this section, we raise an open question. Formulated in a geometric language, this is motivated by the study of whether, in specific instances, Theorem \ref{minspecial} could hold for all (but finitely many) $\u{n} \in \mathbb{Z}^r$. 

    By Theorem \ref{mainlongpaper}, the validity of Theorem \ref{minspecial} for all but finitely many $\u{n} \in \ZZ^r$ amounts to the non-existence of a coset $\mathcal{K}$ of rank $\geq 1$ on which $\mathbf{f}$ is constant. Obviously such a coset exists if and only if there exists one of rank $1$. Since we are mainly interested in (PEP) sets provided by (BG) by semi-simple matrices, it seems natural to ask the following rough question.
    \vskip2mm
	\noindent
    {\bf Problem.}
        Give a characterization of all $r$-tuples of semi-simple matrices $(g_1,\ldots,g_r)\in \GL_n (K)^r$ such that there exists a rank $1$ coset $\mathcal{K} \subseteq \ZZ^r$ for which $g_1^{n_1}\cdots  g_r^{n_r}$ has constant value for $(n_1,\ldots, n_r) \in \mathcal{K}$. 
    
    \vskip2mm

    Writing such a $\mathcal{K}$ as $\u{a}+\ZZ\cdot \u{c}$, where $\mathbf{a}=(a_1,\dots,a_r),\mathbf{c}=(c_1,\dots,c_r)\in \ZZ^r$, then the condition in the problem above amounts to saying that $g_1^{a_1}g_1^{c_1m}\cdots g_r^{a_r}g_1^{c_rm}$ is constant for $m \in \ZZ.$ Letting $M_i:= g_i^{a_i}, g'_i:= g_i^{c_i}, i=1,\ldots, r$, this is equivalent to asking that $M_1\cdot h_1 \cdot M_2 \cdot h_2 \cdots M_r \cdot h_r$ is equal to a constant matrix $M$ as $(h_1,\ldots,h_r)$ varies in the Zariski-closure of $(g'_1,\ldots,g'_r)^{\ZZ} \subseteq \GL_n(K)^r.$ It is well-known that this Zariski-closure is a group of multiplicative type (i.e.\  a finite extension of a finite commutative group by a torus), hence a ``characterization'' as in the question above would be obtained by characterizing all tori lying in the varieties $X:=\{(h_1,\ldots,h_r) \in \GL_n^r \mid M_1h_1\cdots M_rh_rM^{-1}= I_n \}$ (as the $M_i,M$ vary). Since every $1$-dimensional torus in $X$ is contained in a maximal torus, we may restrict to {\bf maximal} tori. Inspired by the above observation, we pose the following more precise question:

    \vskip2mm
	\noindent
    {\bf Question.}
        Let $r \geq 1$ be a natural number, $M_1,\ldots,M_{r+1} \in \GL_n(K)$ be semi-simple matrices, and $X$ be the subvariety of $G:= \GL_n^r$ defined by 
        $$X:=\{(h_1,\ldots,h_r) \in \GL_n^r \mid M_1h_1\cdots M_rh_rM_{r+1}= I_n\}.$$ 
        Let $H\subset \mathrm{Aut}(G)$ be the stabilizer of $X$. Is the set of maximal tori contained in $X$ a finite union of orbits under $H$?
        
    Moreover, one can generalize the above setting and ask the following much more ambitious question. 
    \vskip2mm
    \noindent
	{\bf Question}. Let $G$ be an algebraic group over a field $K$ of characteristic zero, and $X\subset G$ be a subvariety. Let $H\subset \mathrm{Aut}(G)$ be the stabilizer of $X$, give a verifiable sufficient condition on $G$ and $X$ so that the set of maximal tori inside $X$ is contained in a finite union of orbits under $H$.
	\vskip2mm

\vskip5mm
\begin{appendix}
    \section*{Appendix} 
    \renewcommand{\thesection}{A}

    Henceforth $K$ is always a number field. The goal of this appendix is to discuss the ``sparseness'' property of (PEP) sets in linear groups which are not necessarily $S$-arithmetic. To make this precise, we formulate the following question.

	\vskip2mm
	\noindent
	{\bf Question.} Let $\Gamma \subset \mathrm{GL}_n(K)$ be a linear group; then are (PEP) sets in $\Gamma$ sparse in terms of the height? Here we say that a subset $\Sigma \subset \Gamma$ is {\bf sparse} in $\Gamma$ if $\lim\limits_{T\to \infty} \frac{\#\{g\in \Sigma \colon H_{\mathrm{aff}}(g)\leq T\}}{\#\{g\in \Gamma\colon  H_{\mathrm{aff}}(g)\leq T\}}=0.$
	
	\vskip2mm

 We completely answer this question with Theorem \ref{maingpth} below, which strengthens Theorem \ref{secondmainthm}. We only include a sketch of its proof, as this follows the same lines of the one of Theorem \ref{secondmainthm}.

	\begin{theo}[Main group theoretic theorem]\label{maingpth} Let $\Gamma \subset \mathrm{GL}_n(K)$ be a linear group. Consider the three conditions

 (AN) $\Gamma$ is anisotropic; (VA) $\Gamma$ is virtually abelian; (FG) $\Gamma$ is finitely generated.

 Then

		1. If at least one of the conditions (AN),(VA),(FG) fails for $\Gamma$, then all (PEP) sets in $\Gamma$ are sparse;
		
		2. If $\Gamma$ satisfies all of (AN),(VA),(FG), then $\Gamma$ contains a rank $r$ abelian subgroup of finite index, we have
		
		\begin{itemize}
			\item If $r\geq 2$, then some infinite (PEP) sets in $\Gamma$ are sparse, while some are not;
			\item If $r=1$, then none of the infinite (PEP) sets in $\Gamma$ is sparse.
		\end{itemize}
		
	\end{theo}

Note that saying that $\Gamma$ satisfies (AN),(VA),(FG) is just a reformulation of the condition given in point (3) of Theorem \ref{secondmainthm}. Hence:

\begin{coro}
    All (PEP) sets in a linear group $\Gamma \subset \mathrm{GL}_n(K)$ are sparse if $\Gamma$ is not boundedly generated by semi-simple elements.
\end{coro}

\begin{rema}{\rm 
    Let $\Gamma$ be an $S$-arithmetic subgroup of a reductive $K$-group such that $\Gamma^{\mathrm{der}}$ is infinite. Then Theorem \ref{maingpth} may be inferred from combining Theorem \ref{volestimate} and Theorem \ref{firstmainthm}, gives in fact the much stronger estimates:
    \[
    \#{\left\{g \in \Sigma \colon H_{\mathrm{aff}}(g) \leq T\right\}} \asymp (\log T)^{r'} \ll T^{\delta} \ll \#{\left\{g \in \Gamma \colon H_{\mathrm{aff}}(g) \leq T\right\}}.
    \]
    
    In spite of this, the techniques used in this appendix (namely, Evertse's theorem and the theory of generic elements) are substantially different from the ones used to prove Theorem \ref{volestimate} (namely, counting results for lattice points). }
\end{rema}

 To prove Theorem \ref{maingpth}, we need the following.
	
    \begin{lemma}\label{Lemrankequalrank}
        Let $\Sigma_0 \subseteq \GL_n(K)$ be an abelian subgroup generated by finitely many semi-simple matrices, and let $\Sigma=\gamma_0\cdot \Sigma_0\subset \GL_n(K)$ be a translate of $\Sigma _0$  by another semi-simple matrix $\gamma_0$. Then $\rk_{\ZZ}\Sigma_0=\rk_{\PEP}\Sigma$, i.e.\  the rank of $\Sigma_0$ as a $\ZZ$-module is equal to the rank of $\Sigma$ as a (PEP).
    \end{lemma}
    
        Our proof is analog to the ``proof of the reduction step'' at the end of \S 6. However, we represent it here (with more details) for completeness.

       \begin{proof} 
        Let $\gamma_1,\ldots,\gamma_r$ be semi-simple matrices which generate $\Sigma_0$; then $\Sigma_0$ is the image of the (PEP) 
        $$\mathbf{f}_0\colon \ZZ^r \to K,\text{  }\u{n} \mapsto \gamma_1^{n_1}\cdots \gamma_r^{n_r}.$$
         Let $\Lambda \subseteq \overline{K}^*$ denote the subgroup generated by the eigenvalues of the matrices in $\Sigma_0$ ($\Lambda$ is finitely generated as $\Sigma_0$ is commutative of finite rank), so we may choose $\lambda_1,\ldots,\lambda_k$ as a set of generators for $\Lambda$. Moreover, let $g\in \GL_n(\overline{K})$ be a matrix diagonalizing $\Sigma_0$, i.e.\  $g^{-1}\Sigma_0 g\subset D_n$; then we may pick matrices $A_1,\ldots,A_n \in M_{k\times r}(\ZZ)$ such that 
        $$g^{-1}\gamma_1^{n_1} \cdots \gamma_r^{n_r}g=\mathrm{diag}\left(\u{\lambda}^{A_1\cdot \u{n} ^T},\ldots,\u{\lambda}^{A_n\cdot \u{n} ^T}\right)=:D(\u{n}), \text{ where }\u{n}=(n_1,\ldots,n_r).$$
        Then clearly 
        $$h_{\mathrm{aff}}(\gamma_0\cdot g \cdot D(\u{n}) \cdot g^{-1})- h_{\mathrm{aff}}(D(\u{n}))=O(1),$$
        and, on the other hand, $h_{\mathrm{aff}}(D(\u{n}))$ extends to a norm on $(\ZZ^r/\mathcal{K}) \otimes_{\ZZ} \R$ by Proposition \ref{Prop:norm2}, where $\mathcal{K}=\bigcap\limits_i \ker \left(\mathbf{n}\mapsto \boldsymbol{\lambda}^{A_i \cdot \mathbf{n}^T}\right)$. In particular, $\#\{ \u{n} \in \ZZ^r/\mathcal{K} \mid h_{\mathrm{aff}}(f(\u{n})) \leq T \} \sim c \cdot T^{r-\rk_{\ZZ} \mathcal{K}}$ for some $c>0$. Since the matrices $\gamma_1,\ldots,\gamma_r$ commute, the map
        $$\mathbf{f}:\ZZ^r/\mathcal{K} \to \Sigma,\text{  }\u{n} \mapsto \gamma_0 \cdot \mathbf{f}_0(\mathbf{n})= \gamma_0 \gamma_1^{n_1}\cdots \gamma_r^{n_r}$$
        defines a bijection. This implies that $\#\{ \gamma \in \Sigma \mid h_{\mathrm{aff}}(\gamma) \leq T \} \sim c \cdot T^{r-\rk _{\ZZ} \mathcal{K}}$ and hence $\rk_{\PEP}\Sigma=r-\rk _{\ZZ} \mathcal{K}=\rk_{\ZZ}\Sigma_0$.
    \end{proof}
    \begin{rema}\label{valuerprimebg} {\rm Lemma \ref{Lemrankequalrank} shows that when a linear group $\Gamma \subset \GL_n(K)$ over a number field $K$ has (BG) by semi-simple elements, then its rank as a (PEP) set is \underline{equal} to the rank of a(ny) finite index abelian subgroup of $\Gamma$ (whose existence is guaranteed by Theorem \ref{secondmainthm}). }
    
    \end{rema}
    \begin{proof}[Sketch of proof of Theorem \ref{maingpth}]
        We start with case $2$, which is easier to deal with. In this case, $\#\left\{g \in \Gamma \colon H_{\mathrm{aff}}(g) \leq T \right\} \sim c \cdot (\log T)^r$ for some $c >0$ by Lemma \ref{Lemrankequalrank} (applied to the finitely many translates of a finitely generated commutative subgroup of rank $r$ of $\Gamma$). Now, by Theorem \ref{firstmainthm}, any (PEP) set $\Sigma$ (in any affine space) is either finite, or $\# \left\{g \in \Sigma \colon H_{\mathrm{aff}}(g) \leq T \right\} \gg \log T$. Therefore, if $r=1$, we see that all (PEP) subsets of $\Gamma$ are either finite or not sparse. On the other hand, if $r \geq 2$, then any subgroup $\Gamma' < \Gamma$ of rank $1$ provides a sparse (PEP) subset of $\Gamma$ since $\# \left\{g \in \Gamma' \colon H_{\mathrm{aff}}(g) \leq T \right\} \sim c' \cdot \log T, c' >0$; at the same time, $\Gamma$ itself is (PEP) but obviously not sparse in $\Gamma$.
        
        We now turn to the more complicated case $1$. Let $\Sigma \subseteq \Gamma$ be a (PEP) subset of rank $r \geq 0$. We want to prove that $\Sigma$ is sparse in $\Gamma$. Let $\Sigma = \Sigma_1\cup \cdots \cup \Sigma_l$ be a decomposition as in Proposition \ref{decomposepep}. Note that $r=\rk_{\PEP} \Sigma= \max_i \rk _{\PEP} \Sigma_i$. We may thus assume that $\rk _{\PEP} \Sigma_1=r$. Clearly, $\Sigma$ is sparse in $\Gamma$ if and only if $\Sigma_1$ is. So, we may assume $\Sigma=\Sigma_1=\mathbf{f}(\ZZ^r)$, where $\mathbf{f}$ is a reduced (PEP). We divide Case 1 into the following three subcases regarding $\Gamma$, signaled with underlined titles. 
        
        \underline{(AN) fails for $\Gamma$.} Let $\Sigma' \coloneqq \Sigma \cdot \Sigma^{-1}$. Let us first assume, for the sole purpose of simplifying the exposition, that $\Sigma'$ is (PEP). Note that, for instance, this is the case when $\Sigma$ is a BG set by semi-simple elements. Let $g \in \Gamma$ be a non-semi-simple element. Using Lemma \ref{Cor} on $g$ and the (PEP) set $\Sigma'$ we deduce that 
        \begin{equation}\label{Eqappendix}
            \#\{ n \in \ZZ \mid g^n \in \Sigma', |n| \leq T\} \ll (\log T)^a,
        \end{equation}        
        for some $a>0$. It is then fairly easy to see that the set $g^{\ZZ}\cdot \Sigma$ has at least $\frac{\log T}{(\log T)^a} \cdot (\log T)^r$-distribution (when counted by affine height $H_{\mathrm{aff}}$), so does its superset $\Gamma$. In particular, $\Sigma$ is sparse in $\Gamma$. 
        
        If $\Sigma'$ is not (PEP), then it is still the image of a function $\varphi:\ZZ^r \to \A_K^{n^2}$, where each coordinate of $\varphi$ is the quotient of two (PEP) polynomials (the denominators may be chosen to be all equal to the determinant $\det \mathbf{f}(\u{n})$). One can see that, with arguments essentially analogous to those in \S \ref{pfht} of this paper, Theorem \ref{firstmainthm} still holds for the image of such a function $\varphi$, except that \emph{a priori} it could hold with the distribution rate $=O((\log T)^{r'})$ instead of an exact asymptotic.\footnote{The authors actually believe that it should be feasible to use here the result \cite{Levin} of Levin to actually obtain an exact asymptotic. In fact, the main problem for extending our results to vectors of \emph{quotients} of purely exponential polynomials is controlling the greatest common divisor of the numerators and denominators of the coordinates on non-degenerate values of $\u{n}$. This gcd is exactly what Levin controls.} Using this, one manages to verify that \eqref{Eqappendix} is still true even if $\Sigma'$ is not a (PEP), and the rest of the proof is the same.

        \underline{(AN) holds, but (VA) fails  for $\Gamma$.}  Using the theory of generic elements, c.f.\  Proposition 3.5 of \cite{CRRZ}, in an analog way as it was done in the proof of Theorem \ref{secondmainthm}, one finds a semi-simple matrix $\gamma \in \Gamma$ that has an eigenvalue $\lambda$ not equal to a root of unity and satisfies $\langle \lambda\rangle \cap \Lambda =\{1\}$. Then it is not difficult to see that the new (PEP) $$\widetilde{\mathbf{f}}\colon \ZZ^{r+1}\to K^{n^2};\text{  }(n_0,n_1,\ldots,n_r)\mapsto \gamma^{n_0} \cdot \mathbf{f}(n_1,\ldots,n_r)$$
        is also reduced. In particular its image has distribution $\sim c\cdot (\log T)^{r+1}$ by Theorem \ref{firstmainthm} for some $c>0$. Thus $\Sigma$ is sparse in $\widetilde{\mathbf{f}}(\ZZ^{r+1})$ and therefore in $\Gamma \ \left(\supset \widetilde{\mathbf{f}}(\ZZ^{r+1})\right)$ as well.
        
        \underline{(AN), (VA) hold, but (FG) fails  for $\Gamma$.} Then for any $m\in \NN$, we may find a commutative subgroup $\Gamma_m < \Gamma$ of $\ZZ$-rank $m$. Since $\Gamma$ is anisotropic, all elements of $\Gamma_m\subset \Gamma$ are semi-simple, thus $\Gamma _m$ is a (PEP) satisfying $\rk _{\PEP} \Gamma _m=\rk _{\ZZ} \Gamma_m =m$ (according to Lemma \ref{Lemrankequalrank}). By Theorem \ref{firstmainthm}, we then obtain that 
        $$\#\left\{g \in \Gamma \colon H_{\mathrm{aff}}(g) \leq T \right\} \geq \#\left\{g \in \Gamma_m \colon H_{\mathrm{aff}}(g) \leq T \right\} \sim c_m(\log T)^m\text{ as }T\to \infty,$$ 
        for some $c_m>0$. Upon choosing $m >\rk _{\PEP}\Sigma$, we see that $\Sigma$ is sparse in $\Gamma.$
        \end{proof}

    We end the appendix by giving an alternative approach to proving Theorem \ref{maingpth} when $\Gamma$ is non-virtually-abelian:

    \begin{rema}\label{Rmk: T to the delta appendix}
      {\rm Assume $\Gamma< \mathrm{GL}_n(K)$ is non-virtually-abelian. Then $\#\{\gamma \in \Gamma \colon H_{\mathrm{aff}}(\rho(\gamma)) \leq T\} \gg T^\delta$ for some $\delta>0$ because
      \begin{itemize}
          \item if $\Gamma$ contains no unipotent elements, then one argues by Tits' alternative as in Remark \ref{Rmk: a second approach} (one checks that such a $\Gamma$ is non-virtually-solvable).
          \item if $\Gamma$ contains a unipotent element $u$, then this lower bound already holds on the subgroup $\Gamma_1:=\langle u\rangle$.
      \end{itemize}
      Combining with Theorem \ref{firstmainthm}, this (re)proves Theorem \ref{maingpth} for $\Gamma$.}
    \end{rema}
        
    \vskip2mm
	\noindent {{\small {\bf Acknowledgements.} We thank Alex Lubotzky, Amos Nevo, Peter Sarnak and Akshay Venkatesh for their interest and encouragement of this project. We are grateful to Minju Lee and Fei Xu for some helpful communications. The first author is partially funded by the Italian PRIN 2017 ``Geometric, algebraic and analytic methods in arithmetic''. The second author was a guest at the Max Planck Institute for Mathematics whenfirst  working on this article. He thanks the Institute for their hospitality and financial support. Part of this work was done when the fourth author was a member at the IAS, he thanks the support by the National Science Foundation under Grant No. DMS-1926686 and the Ky Fan and Yu-Fen Fan Endowment fund of the Institute for Advanced Study.}}
	
    \end{appendix}
    \bibliographystyle{amsplain}

\end{document}